\documentclass[11pt]{amsart}
\usepackage{mathtools}
\usepackage{amssymb}
\usepackage[margin=1in]{geometry}
\usepackage[utf8]{inputenc}

\usepackage{float} 
\usepackage{graphicx}
\usepackage{epstopdf}
\usepackage{color}
\newtheorem{theorem}{Theorem}
\newtheorem{corollary}[theorem]{Corollary}
\newtheorem{lemma}[theorem]{Lemma}
\newtheorem{proposition}[theorem]{Proposition}
\theoremstyle{definition}

\theoremstyle{remark}
\newtheorem{remark}[theorem]{Remark}

\newcommand{\id}{\mathrm{id}}

\newcommand{\Diff}{\mathrm{Diff}}

\begin{document}

\title{Solar models and McKean's breakdown theorem \\
for the $\mu$CH and $\mu$DP equations}

\author{Stephen C. Preston}

\date{\today}

\address{Department of Mathematics, Brooklyn College and the Graduate Center, City University of New York, NY 11106, USA}
\email{stephen.preston@brooklyn.cuny.edu}

\begin{abstract}
We study the breakdown for $\mu$CH and $\mu$DP equations on the circle, given by
$$m_t + u m_{\theta} + \lambda u_{\theta} m = 0,$$
for $m = \mu(u) - u_{\theta\theta}$, where $\mu$ is the mean and $\lambda=2$ or $\lambda=3$ respectively.
It is already known that if the initial momentum $m_0$ never changes sign, then smooth solutions exist globally.
We prove the converse: if the initial momentum changes sign, then $C^2$ solutions $u$ must break down in finite time.
The technique is similar to that of McKean, who proved the same for the Camassa-Holm equation, but we introduce a
new perspective involving a change of variables to treat the equation as a family of planar systems with central force
for which the conserved angular momentum is precisely the transported vorticity. We also demonstrate how this perspective
can apply to give some insights for other PDEs of continuum mechanics, such as the Okamoto-Sakajo-Wunsch equation
(and in particular the De Gregorio equation).
\end{abstract}

\maketitle

\tableofcontents

\section{Introduction}\label{sectionintro}

In this paper we study the $\mu$-$\lambda$ family of equations
\begin{align}
&m_t(t,\theta) + u(t,\theta) m_{\theta}(t,\theta) + \lambda u_{\theta}(t,\theta) m(t,\theta) = 0,\label{mubgeneral} \\
&m(t,\theta) = \sigma(t) - u_{\theta\theta}(t,\theta), \qquad \sigma(t) = \int_{S^1} u(t,\theta)\,d\theta \label{momentumdef} \\
&u(0,\theta) = u_0(\theta), \qquad t\ge 0, \;\theta\in S^1 = \mathbb{R}/\mathbb{Z}.\label{mubICs}
\end{align}
Here $u(t,\theta)$ is a velocity field on the circle, and $m(t,\theta)$ defined by \eqref{mubICs}
is called its \emph{momentum} or \emph{vorticity}. The two special cases we care about the most are:
\begin{itemize}
\item $\lambda=2$, the $\mu$-Camassa-Holm (or sometimes $\mu$-Hunter-Saxton) equation, and
\item $\lambda=3$, the $\mu$-Degasperis-Procesi equation.
\end{itemize}
Our interest is in whether solutions exist for all time $t\ge 0$, or if they break down at some $T>0$,
given an initial condition $u_0$. We will work with solutions $u(t,\cdot)\in C^2(S^1)$, assuming that
$u_0\in C^2$ and $m_0\in C^0$.

Integrating \eqref{mubgeneral} over $\theta\in S^1$ gives, after an integration by parts, the fact that $\sigma'(t)=0$,
so that for the remainder of the paper we will denote in \eqref{momentumdef}
\begin{equation}\label{sigmadef}
\sigma = \int_0^1 u_0(\theta)\,d\theta.
\end{equation}
If $u_0$ is such that $\sigma=0$ in equation \eqref{sigmadef}, then the breakdown picture is mostly
understood by work of Sarria-Saxton~\cite{SS1,SS2}, who showed that if $\lambda\in [-1,1]$ then all solutions of \eqref{mubgeneral}--\eqref{mubICs}
are global in time; if $1<\lambda\le 5$, then there exist $u_0$ such that solutions break down with $u_{\theta}(t,\theta_*)$ approaching negative
infinity for some $\theta_*\in S^1$; and for all other values of $\lambda$, there is an initial condition such that
breakdown happens everywhere. For $\lambda=2$ with $\sigma = 0$, the equation becomes the Hunter-Saxton equation~\cite{HunterSaxton}, and
its explicit solution together with the geometric interpretation in terms of spherical geodesics were given by Lenells~\cite{Lenells1}.
In particular all solutions break down in finite time with $u_{\theta}\to -\infty$
on a discrete set. If $\lambda=3$ with $\sigma = 0$, the equation \eqref{mubgeneral} is the
second derivative of the inviscid Burgers' equation $u_t + uu_{\theta}=0$, for which all solutions break down
in finite time as pointed out in Lenells-Misio{\l}ek-T\i{\u g}lay~\cite{LMT}. We will review these computations in Section \ref{sectionbackground}.

When $\sigma\ne 0$ the situation is more complicated: for some smooth $u_0$ the solution may break down, while for other smooth $u_0$ the solution
exists globally. Here we settle the question of precisely which initial conditions lead to breakdown for the two simplest and most important
special cases $\lambda=2$ and $\lambda=3$. This theorem is inspired by the result of McKean~\cite{mckeanbreakdown}, who proved the same for the Camassa-Holm
equation, which is \eqref{mubgeneral} but with \eqref{momentumdef} replaced by $m = u - u_{\theta\theta}$. Our proof is inspired by that one, and the
simplified version given in \cite{JNZ}. 

The main novelty of our approach is that we introduce a new central-force model which describes the equation more geometrically. We consider a family of particles in the plane depending on $\theta\in S^1$, such that $\eta_{\theta}(t,\theta)$ is zero if and only if the particle is at the origin. These particles in the plane are subject to a central force, and the conserved angular momentum is precisely the transported vorticity of the Euler-Arnold equation. Unless the central force is sufficiently large, particles with nonzero angular momentum will orbit, like planets in the solar system. However if the angular momentum vanishes, then it is possible (and relatively easy) for a particle to reach the origin in finite time. Thus if the angular momentum is always of the same sign, all particles orbit forever, while if it changes sign, then breakdown can occur. The details still depend on the particular equation, however. 

\begin{theorem}\label{mainthm}
Suppose the initial velocity $u_0$ is $C^2$ on $S^1$, and let $m_0(\theta) = \sigma - u_0''(\theta)$ be the initial momentum.
Assume that either $\lambda=2$ or $\lambda=3$. Then the solution $u$ of \eqref{mubgeneral}--\eqref{mubICs} exists and remains in $C^2$ for all time if and only if $m_0$ never changes sign on $S^1$. If $m_0$ does change sign, then $u_{\theta}(t,\theta_*)$ approaches negative infinity in finite time at a value $\theta_*\in S^1$ where $m_0$ changes from positive to negative.
\end{theorem}

The fact that $m_0\ge 0$ or $m_0\le 0$ everywhere implies global existence is well-known: if $\lambda=2$ it was proven in the original paper of Khesin-Lenells-Misio{\l}ek~\cite{KLM} which introduced the $\mu$CH equation, and if $\lambda=3$ it was proven in the original paper of Lenells-Misio{\l}ek-T\i{\u g}lay~\cite{LMT} which introduced the $\mu$DP equation. We give a different proof which makes a bit more clear geometrically
why this works and generalizes to other equations of the form \eqref{mubgeneral}. On the other hand, while there are several results on sufficient conditions for breakdown of either the $\mu$CH or $\mu$DP equations (see e.g., \cite{FLQ} and \cite{GLZ}), they do not capture all cases. The similarity of Theorem \ref{mainthm} to the result of McKean suggests that a general principle applies: those equations which have the form \eqref{mubgeneral} for some function $m$, given as a pseudodifferential operator in terms of $u$, should have breakdown behavior which depends on the sign of the initial momentum $m_0$. It seems likely that with a bit more work, one can apply the technique here to similar families of PDEs to obtain the complete breakdown picture.

The special cases $\lambda=2$ and $\lambda=3$ in \eqref{mubgeneral}--\eqref{mubICs} are especially interesting because they are both completely integrable, with bihamiltonian structure generating infinitely many conservation laws: see \cite{KLM} and \cite{LMT} respectively. Aside from the conservation of average velocity \eqref{sigmadef}, which is true regardless of $\lambda$, we have for $\lambda=2$ that $\int_{S^1} u_{\theta}(t,\theta)^2 \, d\theta$ is constant, and for $\lambda=3$ that $\int_{S^1} u(t,\theta)^2\,d\theta$ is constant. We will not need any of the other conservation laws, which in general are not coercive. However one can use the complete integrability to obtain the global existence result, as shown in McKean~\cite{mckeanglobal} for the Camassa-Holm equation and sketched in T\i{\u g}lay~\cite{Tiglay} for the $\mu$-Camassa-Holm equation.

In Section \ref{sectionbackground}, we recall the vorticity conservation formula and derive some basic properties of the model \eqref{mubgeneral}--\eqref{mubICs}, including conservation laws. In Section \ref{sectionsolarmodel} we recall the solution formulas for the simplest case of mean-zero velocity fields (for the Hunter-Saxton and Degasperis-Procesi equation) and illustrate the solar model picture of breakdown. In 
Section \ref{sectiontransformation}, we present the general transformation for nonzero $\sigma$ and show that we obtain a central force system, 
where the conserved angular momentum is precisely the vorticity. In 
Section \ref{sectionlocal} we present the local existence theory, showing in particular when $\lambda=3$ that the solution exists in the transformed coordinates up to and slightly beyond the first time a particle reaches the origin; when $\lambda=2$ the solution exists for all time in the transformed coordinates. In Section \ref{sectionboundedforce} we prove that the central force is bounded polynomially in time, and we prove some general aspects of mechanics under central forces (not necessarily coming from a solar model of a PDE). These are used in Section \ref{sectionproofs} to prove Theorem \ref{mainthm}. Finally in Section \ref{sectionoutlook}, we discuss a different transformation of equation \eqref{mubgeneral} (where the momentum is given by $m=Hu_{\theta}$ instead of \eqref{momentumdef}) and illustrate how the solar picture here generates bounds for the solution; this is the Okamoto-Sakajo-Wunsch family of equations, a generalization of the De Gregorio equation which appears in a particularly simple way here. 

The author thanks Martin Bauer, Boris Khesin, Alice Le Brigant, Jae Min Lee, Stephen Marsland, Gerard Misio{\l}ek, Cristina Stoica, Vladimir {\u S}ver{\'a}k, Feride T\i{\u g}lay, and Pearce Washabaugh for very valuable discussions, as well as all the organizers and participants of the BIRS workshop 18w5151 and the Math in the Black Forest workshop for listening to early versions of this work. The work was done while the author was partially supported by Simons Foundation Collaboration Grant \#318969.

\section{Background}\label{sectionbackground}

Equation \eqref{mubgeneral}, for a general $m= L(u)$ defined by a pseudodifferential operator $L$ in terms of $u$, is a generalization of the Euler-Arnold equation. For $\lambda=2$ it is exactly the Euler-Arnold equation: it describes the evolution of geodesics under a right-invariant Riemannian metric on the diffeomorphism group $\Diff(S^1)$ of the circle, where the metric is given at the identity by
\begin{equation}\label{L2norm}
\langle u, u\rangle_{\id} = \int_{S^1} u L u \, d\theta.
\end{equation}
If $L$ is positive-definite, this defines a Riemannian metric, and the actual geodesic in the diffeomorphism group is found by solving the flow equation
\begin{equation}\label{flowequation}
\eta_t(t,\theta) = u\big(t,\eta(t,\theta)\big), \qquad \eta(0,\theta) = \theta.
\end{equation}
Paired with \eqref{mubgeneral}, this is a second-order differential equation for $\eta$; the decoupling is an expression of Noether's theorem due to the right-invariance.
The Camassa-Holm equation with $m=u-u_{\theta\theta}$ is the best-known example in one dimension; in higher dimensions one gets the Euler equations of ideal fluid mechanics and a variety of other equations of continuum mechanics. See surveys in \cite{AK1999}, \cite{KLMP2013}, \cite{KW} for other examples. When $\lambda=2$ and $L$ is nonnegative but not strictly positive, the equation may describe geodesics on quotient spaces of $\Diff(S^1)$, modulo
a quotient group generated by the kernel of $L$; see Khesin-Misio{\l}ek~\cite{KMhomogen} for the requirement. Examples include the Euler-Weil-Petersson equation~\cite{gaybalmazratiu} and the Hunter-Saxton equation.

For other values of $\lambda$, the quadratic form \eqref{L2norm} is not necessarily conserved, and if not then
the equation \eqref{mubgeneral} does not represent the equation for geodesics in a Riemannian metric. However
it can still be interpreted as a geodesic for a right-invariant but non-Riemannian connection; see \cite{KLM} and \cite{EschernonRiem}
for details on this construction in the present cases, and \cite{tiglayvizman} for the general situation.
A well-known example is the Okamoto-Sakajo-Wunsch equation~\cite{OSW}, where
$m = Hu_{\theta}$ in terms of the Hilbert transform $H$ (if $\lambda=-1$ it becomes the well-known De Gregorio
equation~\cite{degregorio}) which are considered the simplest one-dimensional models for vorticity growth in the 3D Euler equation.
We will return to this family at the end of the paper.
On the other hand if $m = -u_{\theta\theta}$ then \eqref{mubgeneral} is the generalized Proudman-Johnson equation,
studied in \cite{SS1,SS2}, which is related to self-similar infinite-energy solutions of the Euler equations of fluids.

What all these equations have in common is the conservation of vorticity property, which we describe as follows. 

\begin{proposition}\label{vorticitytransportprop}
For any equation of the form \eqref{mubgeneral}, regardless of how $m$ is related to $u$, we have the vorticity transport formula 
\begin{equation}\label{vorticitytransport}
\eta_{\theta}(t,\theta)^{\lambda} m\big(t,\eta(t,\theta)\big) = m_0(\theta).
\end{equation}.
\end{proposition}

\begin{proof}
Observe that by the chain rule and the definition \eqref{flowequation} of $\eta$, we have
\begin{equation}\label{transportderivative}
\frac{\partial}{\partial t} m\big(t, \eta(t,\theta)\big) = m_t\big(t, \eta(t,\theta)\big) + u\big(t,\eta(t,\theta)\big) m_{\theta}\big(t,\eta(t,\theta)\big).
\end{equation}
Furthermore differentiating \eqref{flowequation} in $\theta$ yields
\begin{equation}\label{flowderiv}
\eta_{t\theta}(t,\theta) = u_{\theta}\big(t,\eta(t,\theta)\big) \, \eta_{\theta}(t,\theta).
\end{equation}
Using both in \eqref{mubgeneral} shows that
$$ \frac{\partial}{\partial t}\Big( \eta_{\theta}(t,\theta)^{\lambda} m\big(t,\eta(t,\theta)\big) \Big) = 0,$$
which shows that the vorticity $m$ is transported via \eqref{vorticitytransport}. 
This is a consequence only of \eqref{mubgeneral}, and is true regardless of whether $m$ is related to $u$ by \eqref{momentumdef} or not.
\end{proof}

As long as $\eta$ remains a diffeomorphism of the circle, we will have $\eta_{\theta}>0$, so that the sign of $m$ is preserved:
for each $\theta$, the transported vorticity $m\big(t,\eta(t,\theta)\big)$ along the Lagrangian path $\eta(t,\theta)$ is positive
if and only if the initial vorticity $m_0(\theta)$ is positive. Equation \eqref{vorticitytransport} can be inverted to solve for
$u\big(t,\eta(t,\theta)\big)$ in terms of $\eta_{\theta}$ and $m_0$, and from there we may obtain a first-order equation for $\eta$ using \eqref{flowequation}. We will not take this approach directly. Instead we study the second order system \eqref{mubgeneral}--\eqref{mubICs}, \eqref{flowequation} by an approximate linearization. That is, we differentiate \eqref{flowderiv} in time to get a second order equation for
$\eta_{\theta}$, then change variables to simplify it. We will elaborate on the differential geometric meaning of this at the end of the paper.

\begin{proposition}
Suppose $m = \sigma - u_{\theta\theta}$ with the definition \eqref{momentumdef}. Then $\sigma$ is constant, and 
equation \eqref{mubgeneral} can be written in the form
\begin{equation}\label{integrated}
u_{t\theta} + u u_{\theta\theta} + \frac{\lambda-1}{2} u_{\theta}^2 - \lambda \sigma u = I,
\end{equation}
for some function $I$ depending only on time. In addition, if $\lambda=2$ or $\lambda=3$, then $I(t)$ is constant in time. 
\end{proposition}

\begin{proof}
Plugging the formula $m=\sigma-u_{\theta\theta}$ into \eqref{mubgeneral} gives
\begin{equation}\label{explicitmub}
\sigma'(t) - u_{t\theta\theta}(t,\theta) - u(t,\theta) u_{\theta\theta\theta}(t,\theta) + \lambda \sigma(t) u_{\theta}(t,\theta)
- \lambda u_{\theta}(t,\theta) u_{\theta\theta}(t,\theta) = 0.
\end{equation}
Integrate this over $\theta\in S^1$: all terms integrate to zero by periodicity, and we obtain $\sigma'(t)=0$, as mentioned in the Introduction.

Now find the antiderivative in $\theta$ of the remaining terms in \eqref{explicitmub}, and we obtain \eqref{integrated}
for some function $I(t)$. Integrating both sides over the entire circle shows that
\begin{equation}\label{Iformula}
I(t) = \frac{\lambda-3}{2} E(t) - \lambda \sigma^2, \qquad \text{where} \quad E(t) = \int_{S^1} u_{\theta}(t,\theta)^2 \, d\theta.
\end{equation}

Differentiation of \eqref{Iformula}, using \eqref{integrated}, gives
\begin{align*}
E'(t) &= 2 \int_{S^1} u_{\theta}u_{t\theta} \, d\theta \\
&= 2I(t) \int_{S^1} u_{\theta} \, d\theta + \lambda \sigma \int_{S^1} u u_{\theta} \, d\theta - 2\int_{S^1} u u_{\theta}u_{\theta\theta} \,d\theta
- (\lambda-1) \int_{S^1} u_{\theta}^3 \, d\theta \\
&= -(\lambda-2) \int_{S^1} u_{\theta}^3 \, d\theta
\end{align*}
after noticing the first two terms vanish and the third term can be integrated by parts to combine with the fourth term. 

In particular when $\lambda=2$ we have that $E(t)$ is constant, and thus so is $I(t)$. On the other hand, when $\lambda=3$, 
we get $I(t) = -3\sigma^2$, which is constant since $\sigma$ is.
\end{proof}

It is the form \eqref{integrated} of the equation, which makes sense for $u(t,\cdot) \in C^2(S^1)$, that we will view
as fundamental. We will see that the kinetic energy term $E(t)$ defined by \eqref{Iformula} controls the global behavior
of solutions. 
This is precisely the reason why our technique will work well in those two cases, and the lack of a bound on $E(t)$
is the reason we cannot yet prove Theorem \ref{mainthm} for other values of $\lambda$. (As will be clearer later, a polynomial
growth bound for $E(t)$ in $t$ would be sufficient to prove Theorem \ref{mainthm}, but the obvious successive-differentiation
manipulations seem to yield at best exponential growth.)

As is typical with equations of Euler-Arnold type (as first noticed by Ebin-Marsden~\cite{EM};
see also \cite{CKintertial} and \cite{MisClassical}), the equation is best-behaved in terms of
the flow $\eta$, i.e., using the Lagrangian description. To see this here, differentiate \eqref{flowderiv}
with respect to $t$ to get
$$ \eta_{tt\theta}(t,\theta) = \Big( u_{t\theta}\big(t,\eta(t,\theta)\big) + u_{\theta\theta}\big(t,\eta(t,\theta)\big) u\big(t,\eta(t,\theta)\big) +
u_{\theta}\big(t,\eta(t,\theta)\big)^2\Big) \eta_{\theta}(t,\theta).$$
Using this, equations \eqref{integrated}--\eqref{Iformula}, after composing with $\eta$ and using \eqref{flowequation} and \eqref{flowderiv}, become
\begin{equation}\label{densityeqn}
\eta_{tt\theta}  = - \frac{\lambda-3}{2} \, \frac{\eta_{t\theta}^2}{\eta_{\theta}} + \Big[ \lambda \sigma \big( \eta_t(t,\theta)-\sigma\big) + \frac{\lambda-3}{2} E(t)\Big] \eta_{\theta}(t,\theta).
\end{equation}
We are going to view this as an equation for $\eta_{\theta}$, in spite of the fact that $(\eta_t-\sigma)$ must be determined nonlocally by the
spatial integral of $\eta_{\theta}$; this is an unavoidable complication. Now the term in square brackets is relatively easy to control (at least if $\lambda=2$ or $\lambda=3$), while the first term on the right side of \eqref{densityeqn} is of higher order and more likely to become singular. The trick is thus to change variables to eliminate it, and end up with an equation that is nearly linear. We will first analyze this in the simplest case where 
$\sigma=0$ and $\lambda\in\{2,3\}$, and generalize from there.

\section{Solar models for H-S and D-P equations}\label{sectionsolarmodel}


Let us recall the analysis of the equations when $\sigma=0$ and $\lambda=2$ or $\lambda=3$, when everything can be done explicitly. The results here are well-known, but our perspective is new. 
The easiest case is $\lambda=3$ (solved in \cite{LMT}), where \eqref{densityeqn} becomes
$  \eta_{tt\theta}  = 0$. 
Define $x(t,\theta) = \eta_{\theta}(t,\theta)$ and $y(t,\theta) = -\eta_{t\theta}(t,\theta)$. Then we have
$$ x_{tt}(t,\theta) = y_{tt}(t,\theta) = 0,$$
which is a trivial central force system (with no force). Conservation of angular momentum of this system follows from
$$ \frac{\partial}{\partial t} ( x y_t - x_t y) = x y_{tt} - x_{tt} y = 0,$$
and the solutions are given by 
$x(t,\theta) = 1 + t u_0'(\theta)$ and $y(t,\theta) = -t u_0''(\theta)$. These obviously exist for all time, and $x$ remains positive for $t<T = \frac{1}{-\inf_{\theta\in S^1} u_0'(\theta)}$; hence also $\eta_{\theta}=x$ remains positive here. For larger $t$, the function $x(t,\theta)$ becomes negative, which means that $\eta(t,\theta)$ is not invertible as a function of $\theta$: it maps multiple values of $\theta$ to the same point. This leads to our inability to invert the formula $\eta_t(t,\theta) = u\big(t,\eta(t,\theta)\big)$ to find $u$, which is the shock phenomenon: the solution $u$ is not even continuous. Note however that $\eta(t,\theta) = \theta + t u_0(\theta)$ exists and remains as spatially smooth as $u_0$ for all time, another illustration of the fact that things are better in Lagrangian coordinates. 

The more interesting case is $\sigma=0$ and $\lambda=2$. Here equation \eqref{densityeqn} becomes 
\begin{equation}\label{huntersax}
\eta_{tt\theta}  = \frac{1}{2} \, \frac{\eta_{t\theta}^2}{\eta_{\theta}} - \frac{1}{2} E_0 \eta_{\theta}(t,\theta).
\end{equation}
Define $x(t,\theta) = \sqrt{\eta_{\theta}(t,\theta)}$; then equation \eqref{huntersax} becomes 
$$ x_{tt}(t,\theta) = -K^2 x(t,\theta), \qquad K^2 = \frac{E_0}{4}.$$
Here $K$ is constant in both space and time, and we have simple harmonic motion. 
Defining $y(t,\theta) = -2x_{\theta}(t,\theta)$, we clearly also have
$$ y_{tt}(t,\theta) = -K^2 y(t,\theta).$$
Since $x_t(t,\theta) = \tfrac{1}{2} \eta_{t\theta}(t,\theta) \eta_{\theta}(t,\theta)$ and $y_t(t,\theta) = -2x_{t\theta}(t,\theta)$, the fact 
that $\eta(0,\theta)=\theta$ and $\eta_t(0,\theta) = u_0(\theta)$ yields the initial conditions 
\begin{align*}
x(0,\theta) &= 1, &\qquad x_t(0,\theta) &= \tfrac{1}{2} u_0'(\theta) \\
y(0,\theta) &= 0, &\qquad y_t(0,\theta) &= -u_0''(\theta) = m_0(\theta)
\end{align*}
The solutions with these initial conditions are
$$ x(t,\theta) = \cos{Kt} + \tfrac{u_0'(\theta)}{2K} \sin{Kt}, \qquad y(t,\theta) = -\tfrac{u_0''(\theta)}{K} \sin{Kt}.$$
We can easily see that $x$ remains positive for
$$t<T = \frac{1}{K}\, \arctan{\left( \frac{2K}{\inf u_0'(\theta)}\right)}$$
and becomes negative beyond that. However since $\eta_{\theta}(t,\theta) = x(t,\theta)^2$ in this case, we will find for typical initial data
that $\eta_{\theta}(t,\theta)$ is positive for all $\theta$ except a discrete set of points (depending on $t$), which means $\eta$ will be a homeomorphism even if it not a diffeomorphism. This allows us to define $u$ as a continuous function, although its derivative $u_{\theta}$ will approach negative infinity wherever $x(t,\theta)=0$ by \eqref{flowderiv}. 
Note that again the central force system has conserved angular momentum, now given explicitly by 
$$ x(t,\theta) y_t(t,\theta) - y(t,\theta) x_t(t,\theta) = -u_0''(\theta) = m_0(\theta).$$
This  is the reason for the scaling on $y$. 
In Figure \ref{b2breakdownfig} we demonstrate what this looks like for a simple solution of the Hunter-Saxton equation.

\begin{figure}[!ht]
\begin{center}
\includegraphics[scale=0.35]{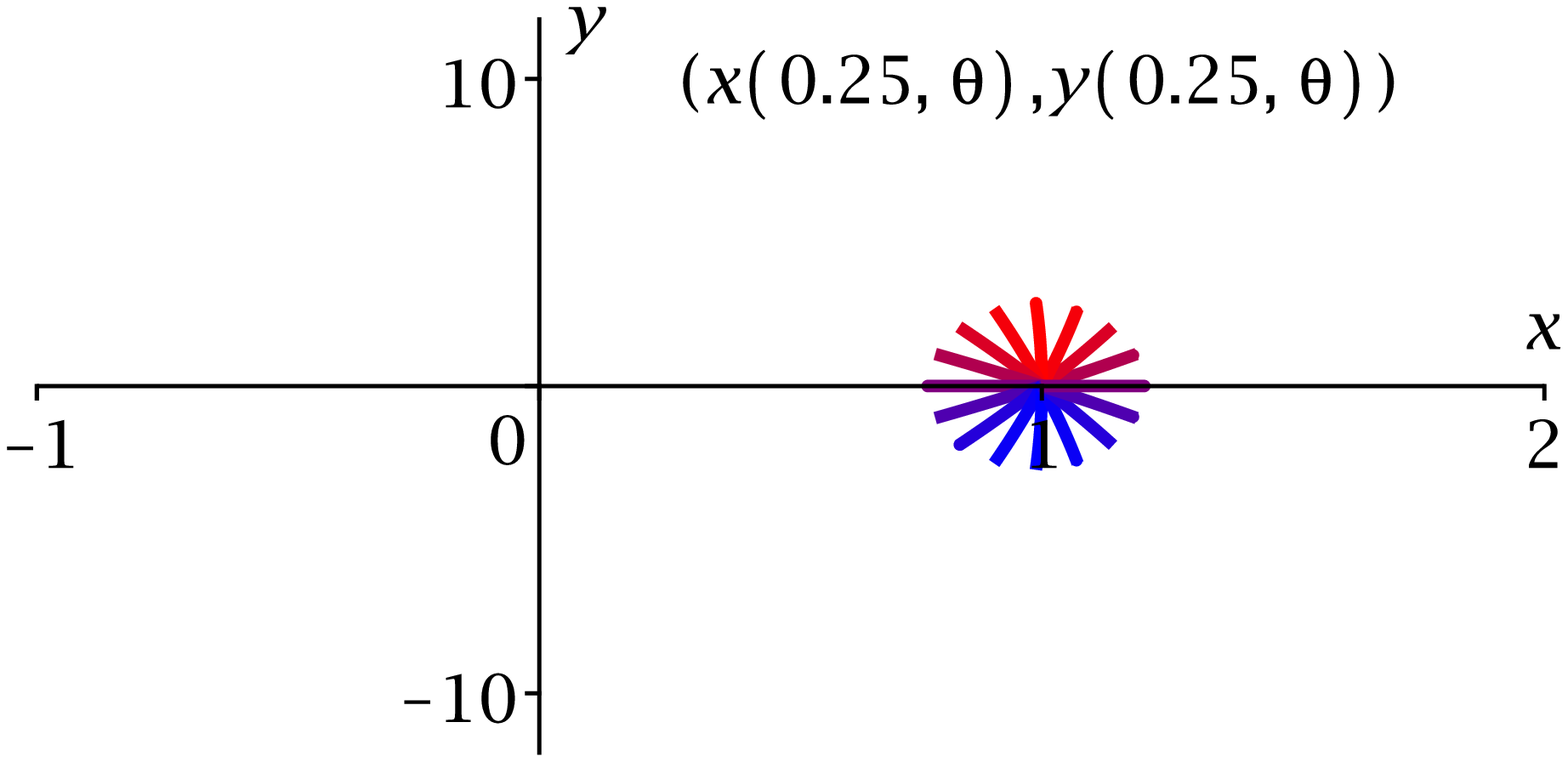} \qquad \includegraphics[scale=0.35]{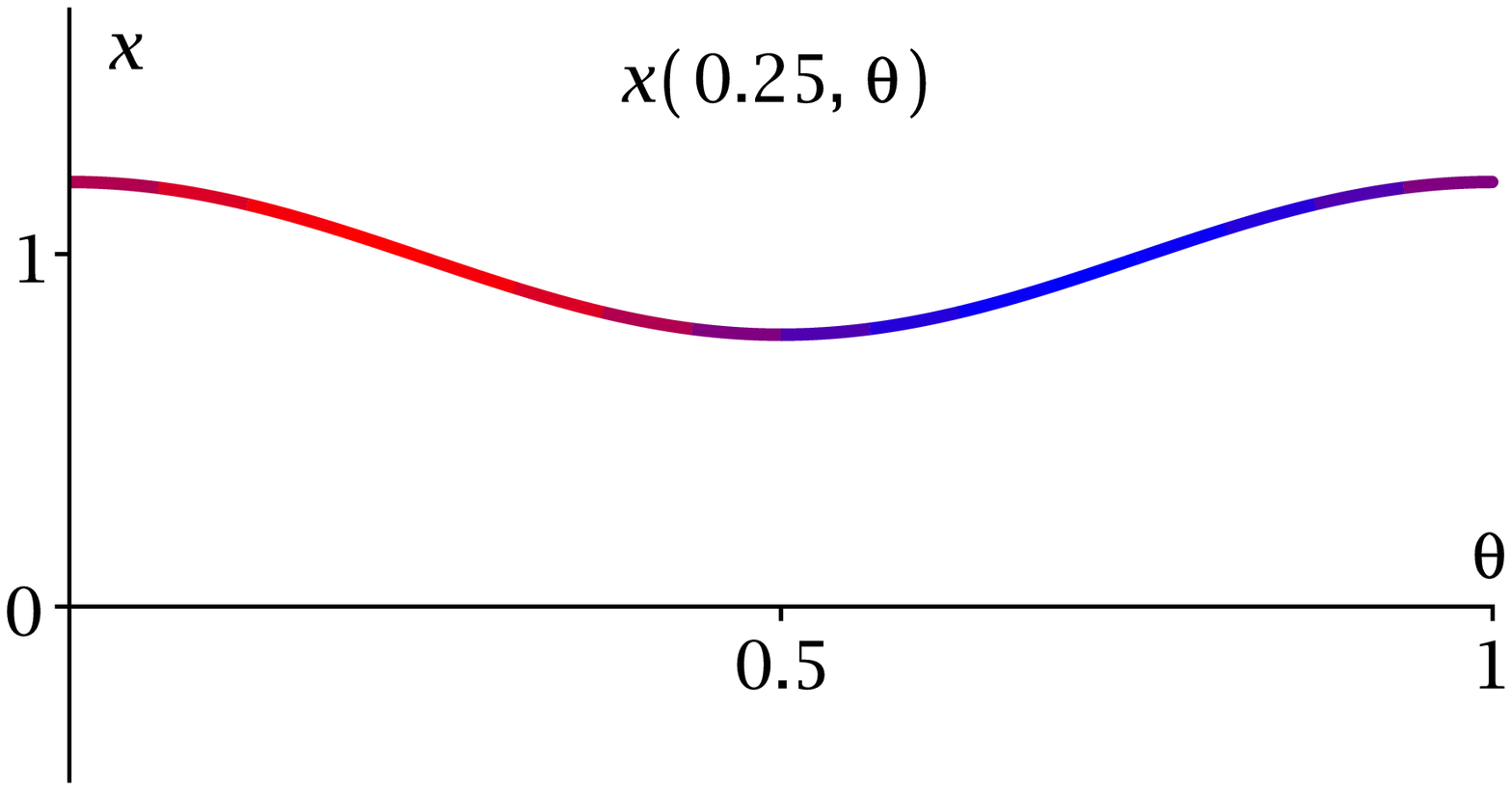} \\
\includegraphics[scale=0.35]{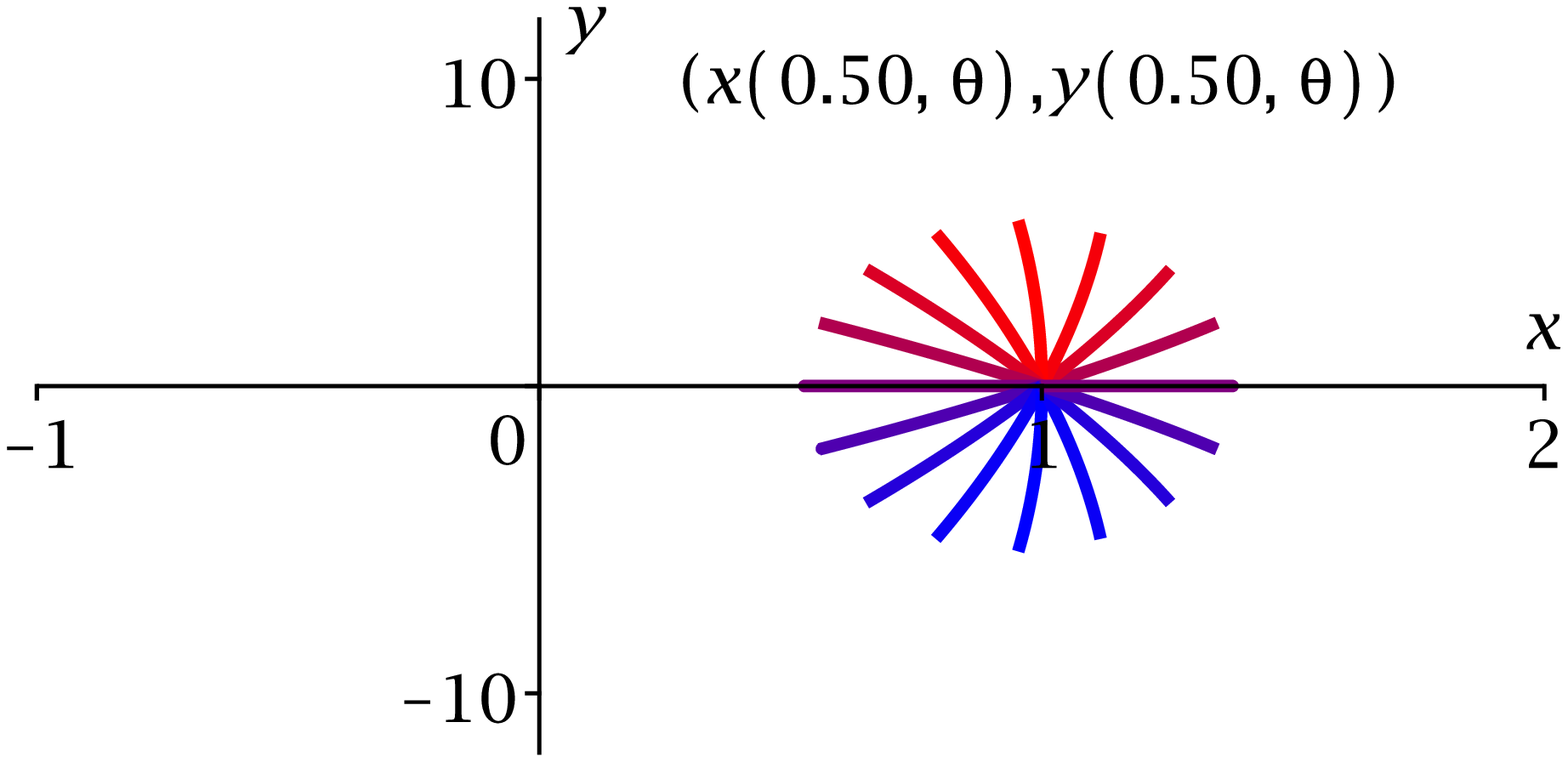} \qquad  \includegraphics[scale=0.35]{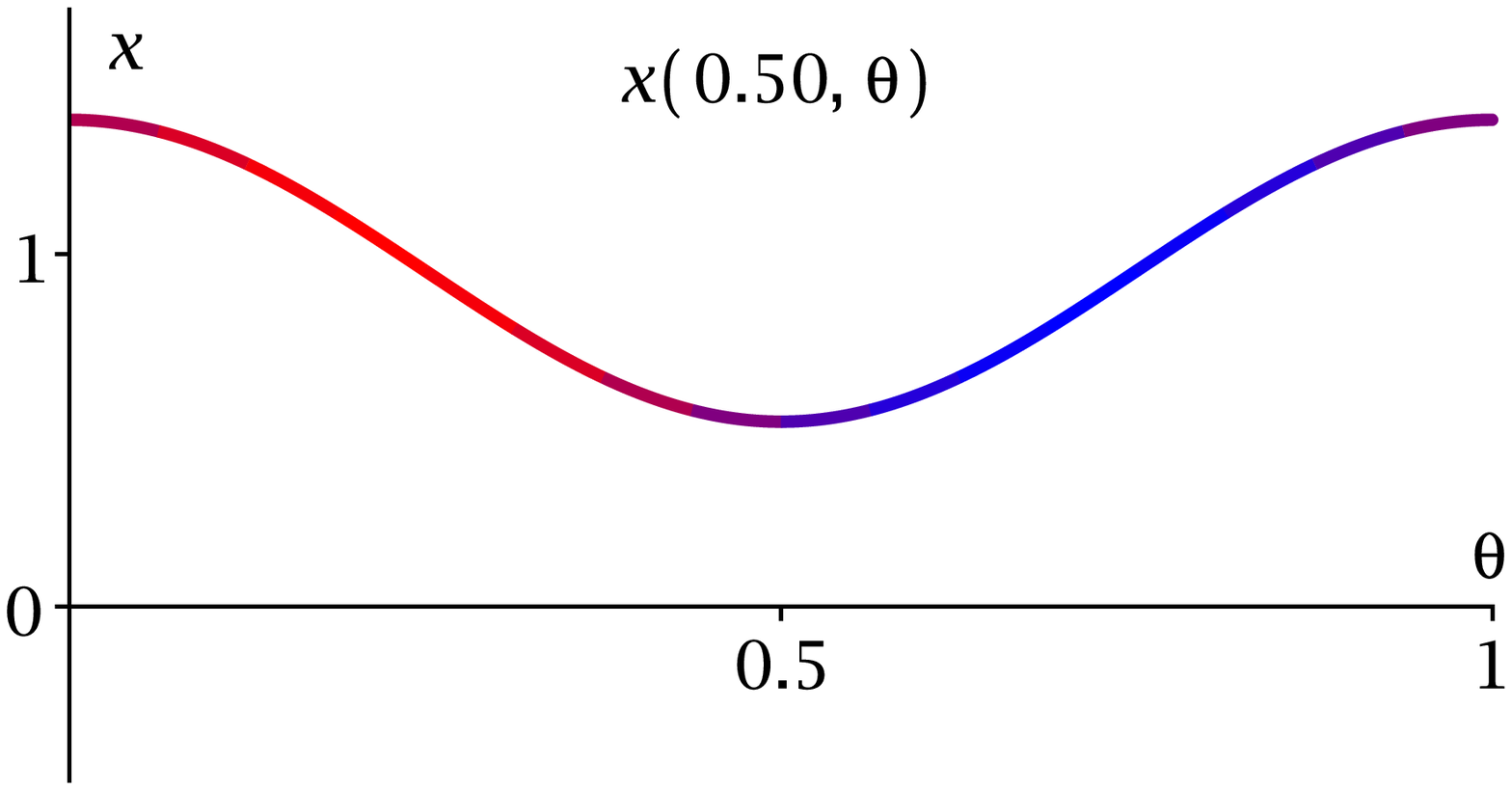} \\
\includegraphics[scale=0.35]{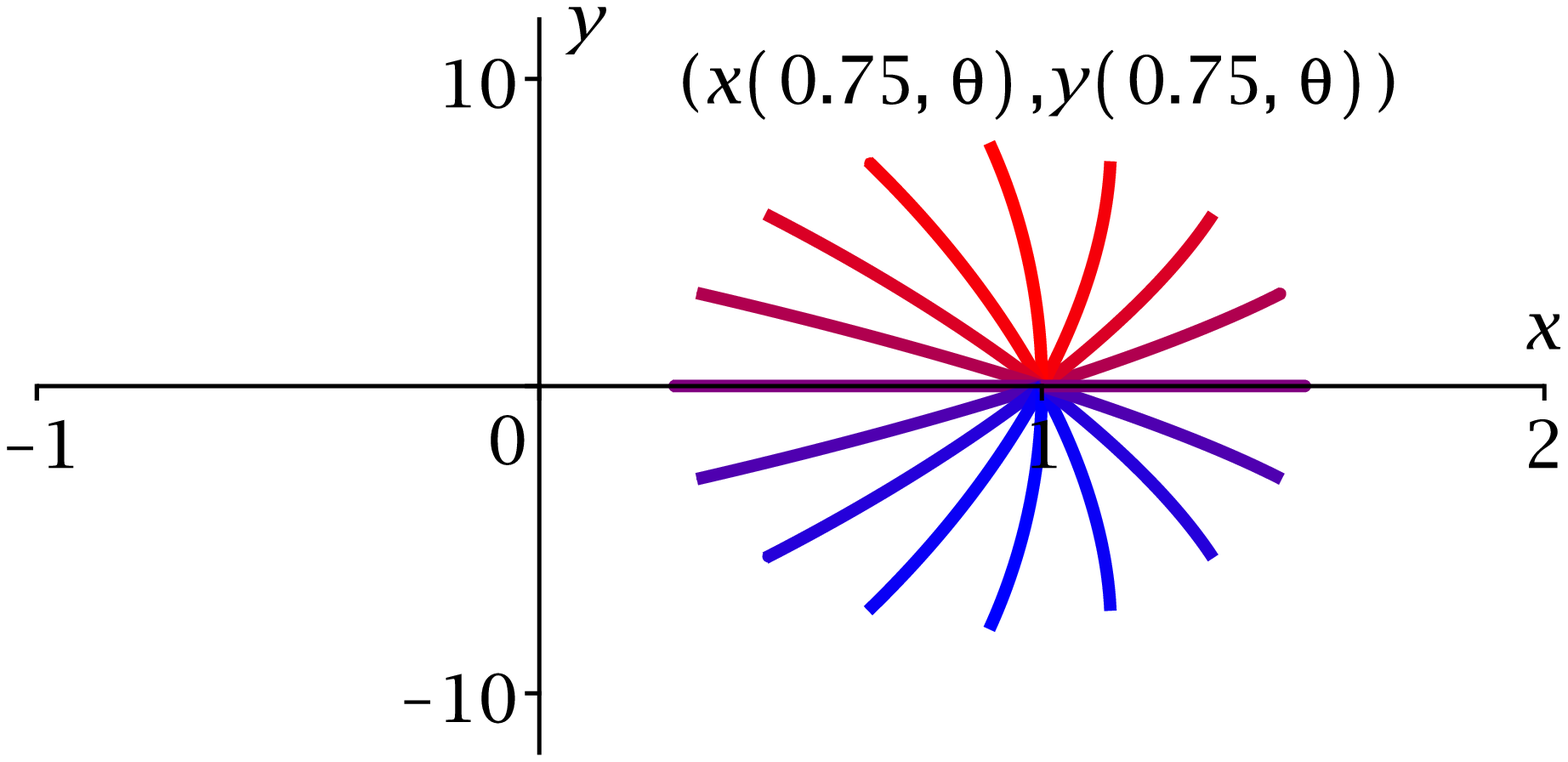} \qquad \includegraphics[scale=0.35]{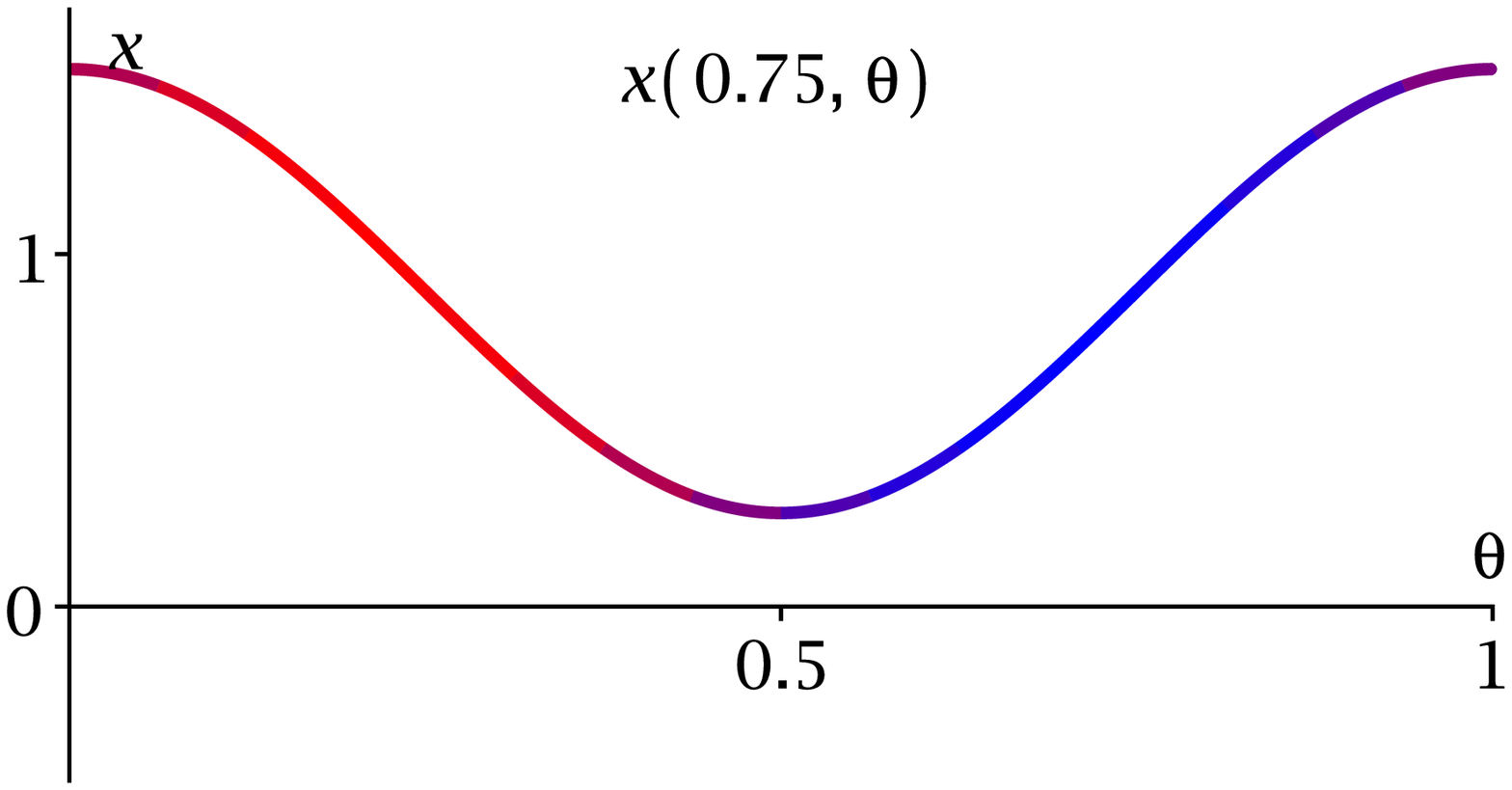} \\
\includegraphics[scale=0.35]{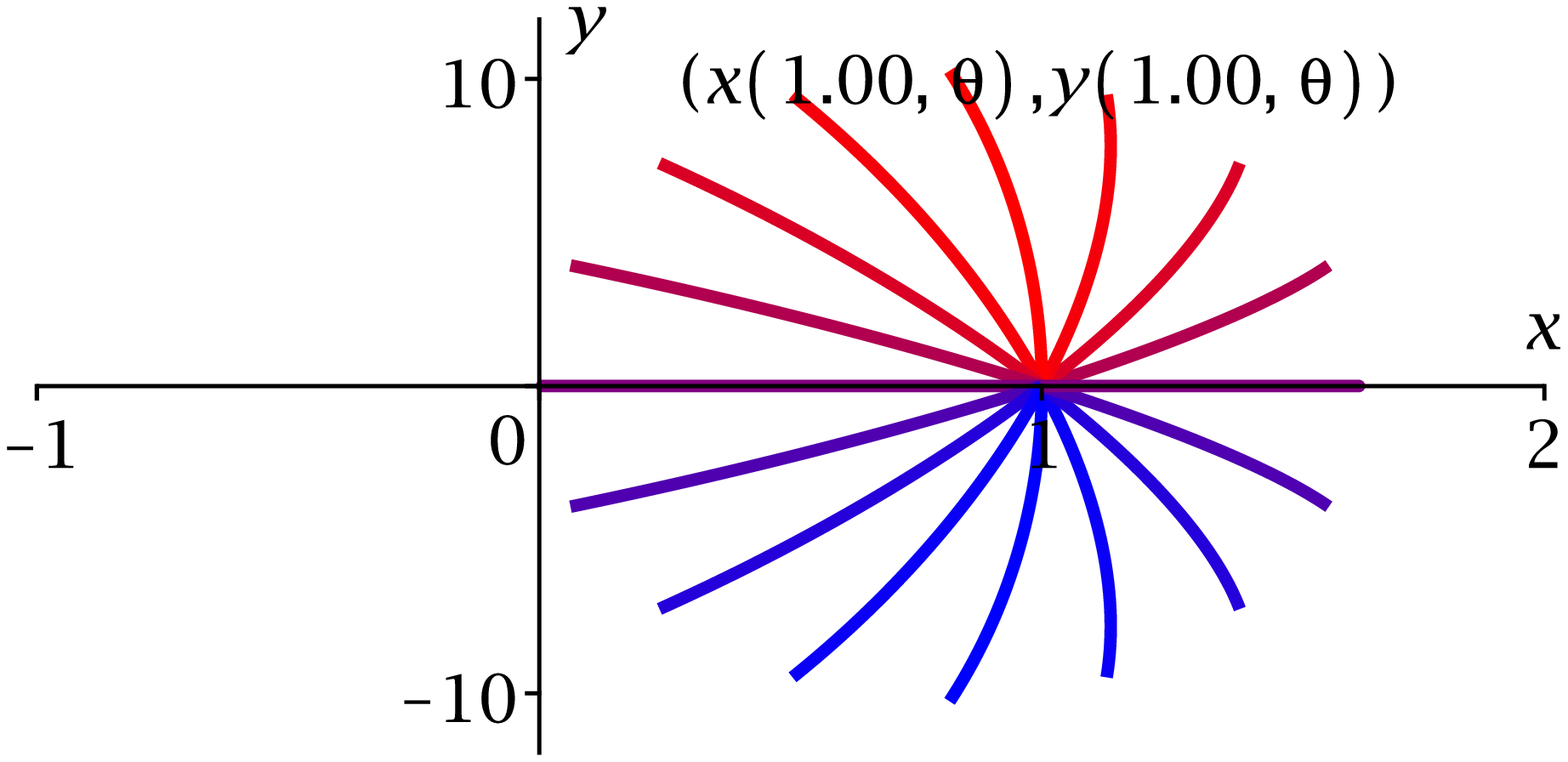} \qquad \includegraphics[scale=0.35]{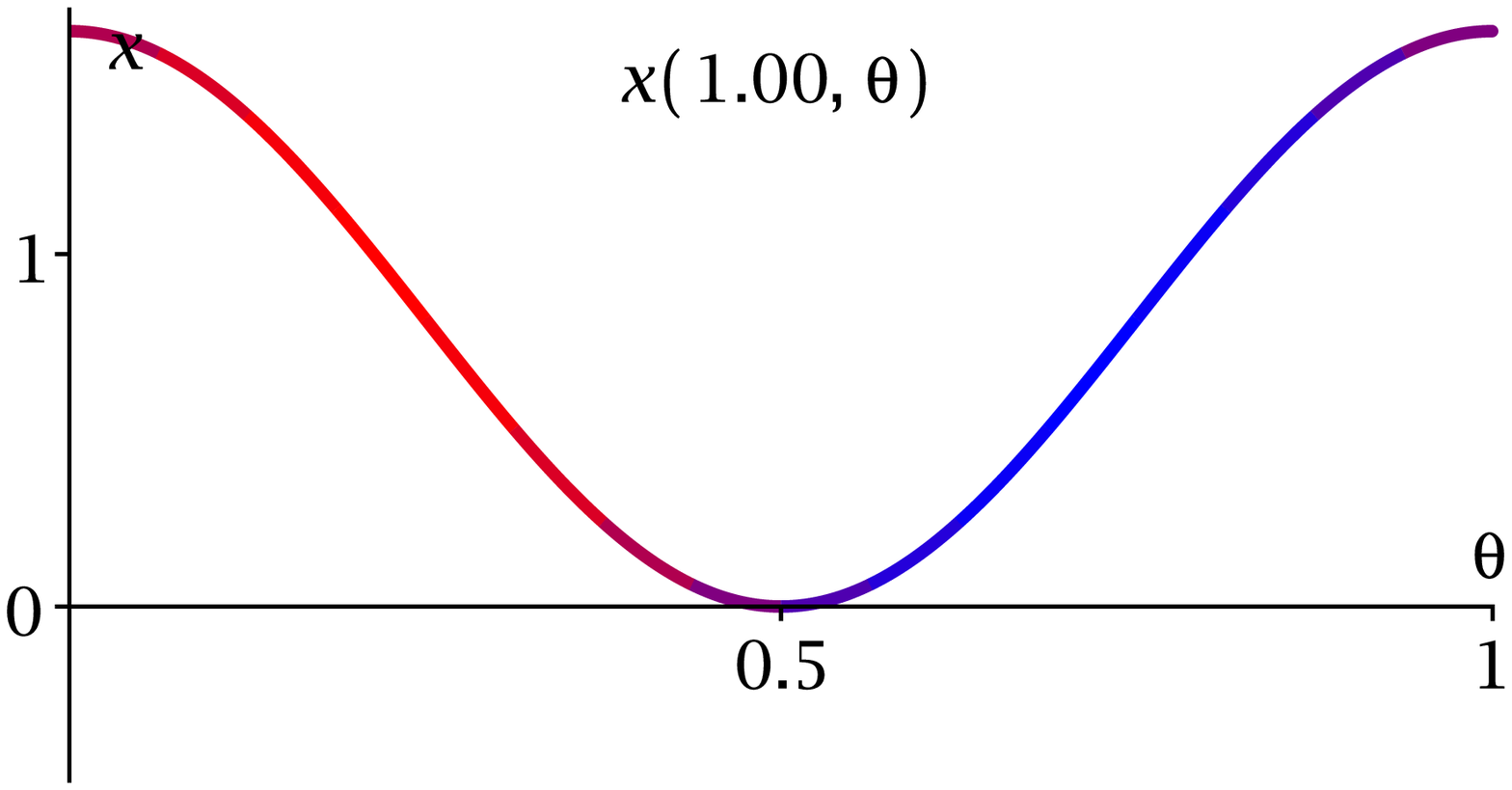} \\
\includegraphics[scale=0.35]{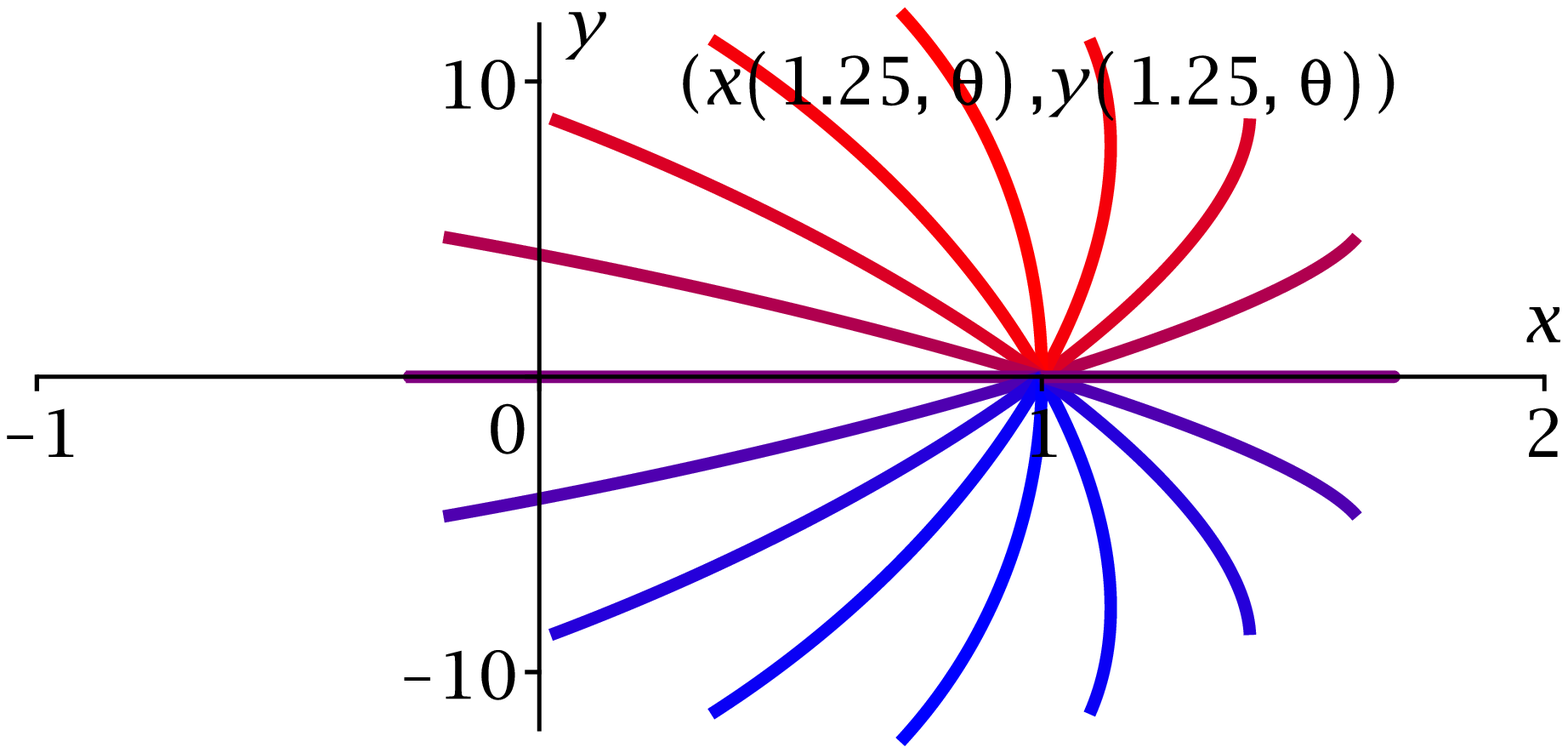} \qquad \includegraphics[scale=0.35]{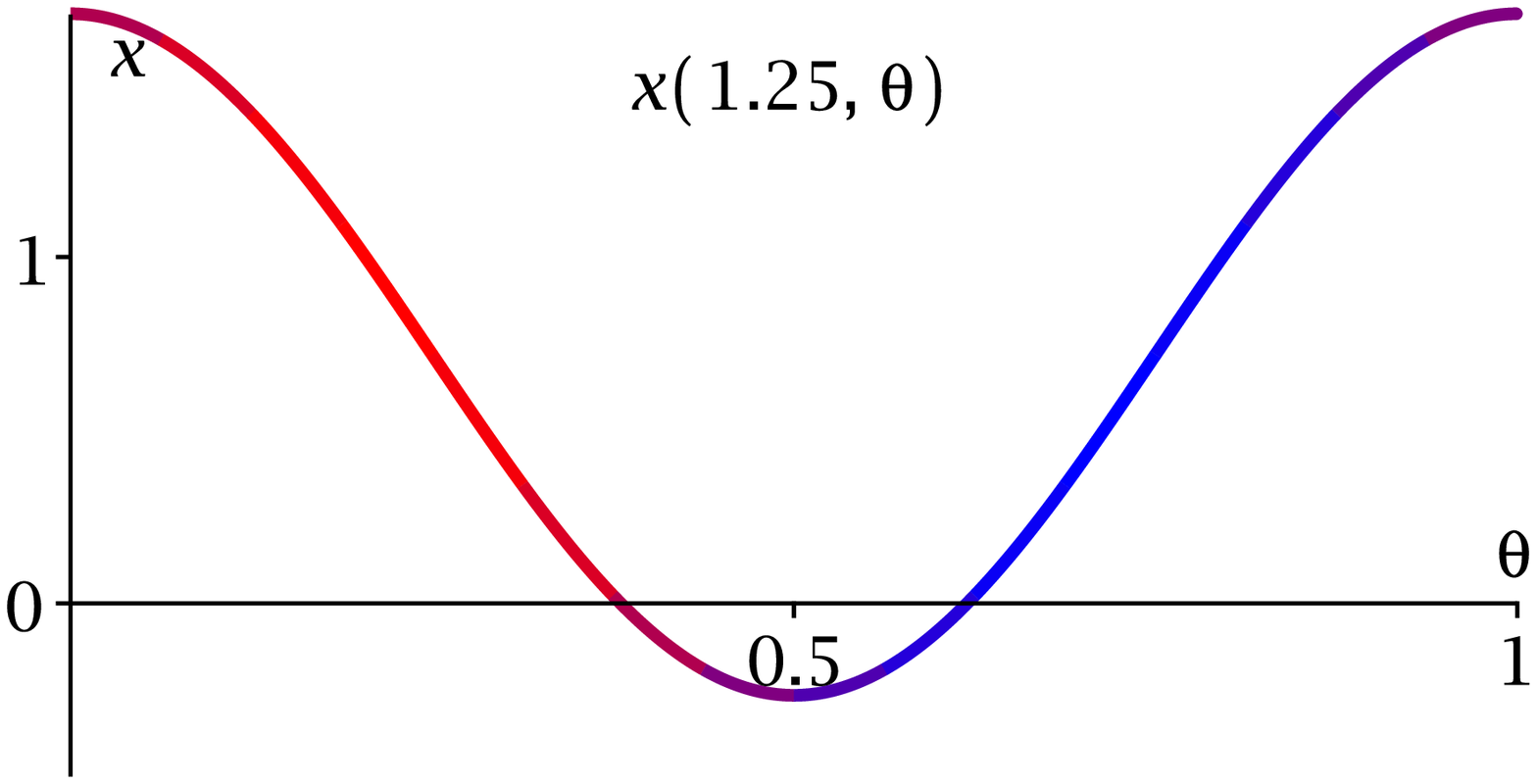} \\
\end{center}
\caption{Here we show both the solar model on the left and the solution $x(t,\theta)=\sqrt{\eta_{\theta}(t,\theta)}$ on the right for the Hunter-Saxton equation, with initial condition $u_0(\theta) = \alpha \sin{(2\pi \theta)}$ for $\alpha = \tfrac{2}{\pi} \arctan(\tfrac{1}{\sqrt{2}})$, with a breakdown time of $t=1$. In the solar model particles emerge from $(1,0)$ with velocity $\langle \tfrac{1}{2}u_0'(\theta),\omega_0(\theta)\rangle$ and approach the vertical wall $x=0$. On the right $x$ and $y=-2x_{\theta}$ have simultaneously reached zero, and the classical solution $u(t,\theta)$ breaks down. However the solution continues in the $(x,y)$ variables. Points colored red have positive angular momentum, while those in blue have negative angular momentum: the first breakdown occurs at the transition.}\label{b2breakdownfig}
\end{figure}

\begin{remark}\label{2notequal3remark}
We see that breakdown is very different already between $\lambda=2$ and $\lambda=3$. One might have expected that since $\lambda$ only appears as a coefficient of lower-order terms in the PDE \eqref{integrated}, it does not have a large role in the breakdown picture. However if $\lambda=2$ we have global weak solutions $u$ which remain continuous (and the corresponding $\eta$ typically remains a homeomorphism even if it is not a diffeomorphism). In 
fact if we consider all weak solutions that conserve energy, the family found here is unique~\cite{Tiglay2}. On the other hand if $\lambda=3$, the solution $u$ must become discontinuous, and as is well known the solution is no longer unique without an extra entropy condition. 
\end{remark}

\section{The general transformation}\label{sectiontransformation}

In the cases of the last section, we have seen that for each fixed $\theta$, the functions $x(t,\theta)$ and $y(t,\theta)$ form the components of a central-force system, which implies that the
angular momentum is always conserved. This conserved quantity is precisely the transported vorticity, so that the conservation law \eqref{vorticitytransport} is
encoded here automatically. This fact is what ensures that when the vorticity is always positive or always negative, classical solutions will be global; see
Theorem \ref{globalposmom}. The intuition is that the $(x,y)$ system is attracted or repulsed by a central force, analogously to the sun's gravity, and singularities correspond to the particle reaching the sun in finite time. As in our solar system, this can only happen if the particle dives directly into it, and any nonzero angular momentum prevents this. A very singular force may still lead to finite-time collapse, but in our situations the force is bounded on finite time intervals. We will now show how to obtain this picture in the general case when $\sigma\ne 0$ and $\lambda$ is any real number.

\begin{theorem}\label{transformthm}
For a parameter $\lambda\ne 1$, define $\gamma = \frac{2}{\lambda-1}$. Set
\begin{equation}\label{xydef}
x(t,\theta) = \eta_{\theta}(t,\theta)^{1/\gamma} \qquad\text{and}\qquad y(t,\theta) = -\gamma x_{\theta}(t,\theta) + \sigma x(t,\theta) \int_0^t x(\tau,\theta)^{\gamma}\,d\tau.
\end{equation}
Then the equation \eqref{densityeqn} is equivalent to the pair of equations
\begin{align}
\frac{\partial^2 x}{\partial t^2}(t,\theta) =
F(t,\theta) x(t,\theta)  \label{xeq}\\
\frac{\partial^2 y}{\partial t^2}(t,\theta) = F(t,\theta) y(t,\theta), \label{yeq}
\end{align}
with
\begin{equation}\label{Fdef}
F(t,\theta) = \frac{\lambda(\lambda-1)\sigma}{2}\, G(t,\theta) + \frac{(\lambda-1)(\lambda-3)}{4} E(t),
\end{equation}
where $E(t)$ defined by \eqref{Iformula} becomes
\begin{equation}\label{Edef}
E(t) = \gamma^2 \int_0^1 x(t,\phi)^{\gamma-2} x_t(t,\phi)^2 \, d\phi
\end{equation}
and $G(t,\theta):= \eta_t(t,\theta)-\sigma$ is given by
\begin{equation}\label{Gdef}
G(t,\theta) = \int_0^{\theta} x(t,\phi)^{\gamma-1} x_t(t,\phi)\,d\phi - \int_0^1 x(t,\phi)^{\gamma} \int_0^{\phi} x(t,\psi)^{\gamma-1} x_t(t,\psi)\,d\psi\,d\phi.
\end{equation}
The initial conditions for these equations are given by
\begin{alignat}{3}
x(0,\theta) &= 1, &\qquad x_t(0,\theta) &= \tfrac{1}{\gamma} u_0'(\theta) \label{xIC}\\
y(0,\theta) &= 0, &\qquad y_t(0,\theta) &= \sigma - u_0''(\theta) = m_0(\theta) \label{yIC}
\end{alignat}
\end{theorem}

\begin{proof}
Since $\int_0^1 \eta_{\theta}(t,\theta)\,d\theta = 1$ for all $t$, note that we always have 
\begin{equation}\label{xpowerintegral}
\int_{S^1} x(t,\theta)^{\gamma}\,d\theta = 1.
\end{equation}

The formula \eqref{xeq} is a straightforward computation from \eqref{densityeqn}: the transformation $\eta_{\theta} = x^{\gamma}$ gives
$$ \eta_{tt\theta} + \frac{\lambda-3}{2} \,\frac{\eta_{t\theta}^2}{\eta_{\theta}} = \gamma x^{\gamma-1} x_{tt} + \gamma\left( \frac{\gamma(\lambda-1)}{2} - 1\right) x^{\gamma-2} x_t^2,$$
so that $\gamma = \frac{2}{\lambda-1}$ eliminates the quadratic term $x_t^2$ from the equation.
We then obtain
\begin{equation}\label{xequationstep}
x_{tt}(t,\theta) = \frac{\lambda-1}{2} \Big[ \lambda\sigma \big( \eta_t(t,\theta) - \sigma\big) + \frac{\lambda-3}{2} \, E(t)\Big] x(t,\theta).
\end{equation}
The formula for $G(t,\theta)$ is determined from the fact that we know
\begin{equation}\label{Gderivative}
G_{\theta}(t,\theta) = \eta_{t\theta}(t,\theta) = \gamma x(t,\theta)^{\gamma-1} x_t(t,\theta)
\end{equation}
as well as the fact that
\begin{equation}\label{Gintegral}
\int_{S^1} G(t,\theta) \eta_{\theta}(t,\theta)\,d\theta = 0,
\end{equation}
and these two conditions clearly uniquely determine $G$. The condition \eqref{Gintegral}
comes from the change of variables formula and \eqref{sigmadef}: we have
\begin{align*}
0 &= \int_{S^1} \big[ u(t,\phi)-\sigma\big] \, d\phi = \int_{S^1} \Big[ u\big(t,\eta(t,\theta)\big) - \sigma\Big] \eta_{\theta}(t,\theta) \, d\theta \\
&= \int_{S^1} \Big[ \eta_t(t,\theta) - \sigma\Big] \eta_{\theta}(t,\theta)\,d\theta.
\end{align*}
We can easily compute that $G$ defined by formula \eqref{Gdef} satisfies both requirements, using the formula \eqref{xpowerintegral}, and so \eqref{xequationstep} becomes \eqref{xeq}.

To prove \eqref{yeq}, we differentiate the formula \eqref{xydef} defining $y(t,\theta)$ twice with respect to time and obtain
$$ y_{tt}(t,\theta) = -\gamma x_{tt\theta}(t,\theta) + \sigma x_{tt}(t,\theta) \int_0^t x(\tau,\theta)^{\gamma}\, d\tau + (\gamma+2) \sigma x(t,\theta)^{\gamma} x_t(t,\theta).$$
Now insert the equation $x_{tt} = F x$, and its spatial derivative, to get
$$ y_{tt}(t,\theta) = F(t,\theta) y(t,\theta) -\gamma F_{\theta}(t,\theta) x(t,\theta) +  (\gamma+2) \sigma x(t,\theta)^{\gamma} x_t(t,\theta).$$
The last two terms in this equation cancel out using \eqref{Fdef} and \eqref{Gderivative}, which produces \eqref{yeq}.

The initial conditions come from the fact that $\eta(0,\theta) = \theta$ so that $\eta_{\theta}(0,\theta)\equiv 1$, which gives the conditions for $x(0,\theta)$ and $y(0,\theta)$. Differentiating the formula \eqref{xydef} with respect to $t$ and using \eqref{flowderiv} gives $\gamma x_t(0,\theta) = u_0'(\theta)$, along with $y_t(0,\theta) = -\gamma x_{t\theta}(0,\theta) + \sigma$, which is exactly the initial momentum $m_0(\theta) = \sigma - u_0''(\theta)$.
\end{proof}

The forcing term $F(t,\theta)$ defined by \eqref{Fdef} appearing in \eqref{xeq}--\eqref{yeq} depends on the solution $x$ and $x_t$ (or if we like on $y$ and $y_t$, since we can in principle reconstruct $x$ from $y$ if desired). As such we properly view \eqref{xeq} as an ODE on a Banach space. Fortunately the dependence of $F$ on $x$ and $x_t$ is relatively simple, and is well-behaved even if $x$ has only limited smoothness---for example if $x(t,\cdot)$ and $x_t(t,\cdot)$ are in $C^k(S^1)$ for some integer $k\ge 0$, then the function $F(t,\cdot)$ will be in $C^{k+1}(S^1)$. More importantly, the map $\Psi := (x,x_t)\mapsto F$ from $C^k\times C^k\to C^{k+1}$ is actually $C^{\infty}$ as a map of Banach spaces as long as $x$ remains positive (which is only needed for the power function to be smooth). Hence equation \eqref{xeq} describes a $C^{\infty}$ ODE on the space of functions $x$ satisfying
\begin{equation}\label{intxpower}
x\in C^k(S^1), \qquad \int_{S^1} x(\theta)^{\gamma} \, d\theta = 1, \qquad x(\theta)>0 \quad \forall \theta\in S^1,
\end{equation}
where the integral condition comes from \eqref{xpowerintegral}.
If $\gamma = \frac{2}{\lambda-1}$ happens to be an integer, as it does for $\lambda=2$ and $\lambda=3$, we get smoothness even for functions $x$ that may be zero or negative at some points, and this allows us to extend the ODE to the larger space
$$ x\in C^k(S^1), \qquad \int_{S^1} x(\theta)^{\gamma}\,d\theta = 1.$$
As we are interested in the breakdown of the equation when $\eta_{\theta}\to 0$, allowing $x$ to approach zero (and even continue to go negative) gives us global solutions in the new coordinate, which translate into weak solutions when we invert to get $\eta_{\theta}$, and from this $\eta$ and $u$.

\begin{corollary}\label{angularmomentumcor}
The angular momentum of the system \eqref{xeq}--\eqref{yeq} is conserved, and given by the formula
\begin{equation}\label{angularmomentum}
x(t,\theta)y_t(t,\theta) - y(t,\theta) x_t(t,\theta) = \eta_{\theta}(t,\theta)^{\lambda} m\big(t,\eta(t,\theta)\big) = m_0(\theta).
\end{equation}
\end{corollary}

\begin{proof}
The fact that angular momentum is conserved for central force systems is well-known: it follows from
$$ \frac{\partial}{\partial t} (xy_t - yx_t) = x y_{tt} - y x_{tt} = x(Fy) - y(Fx) = 0.$$
Equation \eqref{xydef} implies that
$$ \frac{\partial}{\partial t} \left( \frac{y(t,\theta)}{x(t,\theta)}\right) = -\gamma \, \frac{\partial}{\partial t} \left(\frac{x_{\theta}(t,\theta)}{x(t,\theta)}\right) + \sigma x(t,\theta)^{\gamma},$$
so that
\begin{align*}
x(t,\theta)y_t(t,\theta) - y(t,\theta)x_t(t,\theta) &=
-\gamma \, x(t,\theta)^2 \, \frac{\partial^2}{\partial t\partial \theta} \left(\ln{\big(x(t,\theta)\big)}\right) + \sigma x(t,\theta)^{\gamma+2}  \\
&= -\eta_{\theta}(t,\theta)^{\lambda-1} \, \frac{\partial^2}{\partial t\partial \theta} \left(\ln{\big(\eta_{\theta}(t,\theta)\big)}\right) +\sigma x(t,\theta)^{\lambda\gamma} \\
&= -\eta_{\theta}(t,\theta)^{\lambda-1} \, \frac{\partial}{\partial \theta} \left(u_{\theta}\big(t,\eta(t,\theta)\big)\right) +\sigma \eta_{\theta}^{\lambda}(t,\theta) \\
&= \eta_{\theta}(t,\theta)^{\lambda} \Big( \sigma - u_{\theta\theta}\big(t,\eta(t,\theta)\big)\Big).
\end{align*}
At time $t=0$, the right side is $m_0(\theta)$.
\end{proof}

\section{Local and global existence in the transformed variables}\label{sectionlocal}

Because the transformation to Lagrangian coordinates eliminates the loss of derivatives (essentially just being
able to combine terms like $m_t + u m_{\theta}$ into $\frac{\partial}{\partial t} m\circ\eta$ as in equation 
\eqref{transportderivative}), we get a smooth ODE
on the space of functions $(x,y)$. We want to work in the simplest space for which all the functions make sense, so we
will require that $u_0$ be $C^2$ in order to have the momentum be continuous. We then expect $u(t,\cdot)$ to be in $C^2$ for
short time, which by the flow equation \eqref{flowequation} should imply that $\eta$ is also spatially in $C^2$; hence
$x(t,\cdot)$ would be in $C^1$ and $y(t,\cdot)$ would be in $C^0$. Working in these spaces, we thus get existence of
solutions using Picard iteration. The following was proved for the case $\lambda=2$ by Deng-Chen~\cite{DC}, following the technique
of Lee~\cite{Lee} for the Camassa-Holm equation.
The proof for other values of $\lambda$ is similar, and just involves showing that
$F$ defined by \eqref{Fdef} is smooth as a function of $x$ and $x_t$.

\begin{theorem}\label{localexistence}
Consider the situation in Theorem \ref{transformthm}. The equation \eqref{xeq} is a second-order smooth ODE on the manifold
$$ \mathcal{S}^1_{\gamma} = \left\{ x\in C^1(S^1) \, \big\vert\, x(\theta)>0 \;\forall \, \theta\in S^1, \; \int_{S^1} x(\theta)^{\gamma} \, d\theta = 1 \right\}.$$
As such, for each initial condition $x(0)\equiv 1$ and $\frac{dx}{dt}(0) = \tfrac{1}{\gamma} u_0'(\theta)$ with $u_0\in C^2(S^1)$,
there is a $T>0$ and a solution $x\colon [0,T) \to C^1(S^1)$ of equation \eqref{xeq}.
\end{theorem}

\begin{proof}
The main point is to write it as a first-order system with $v:=x_t$, viewing $E$, $F$, and $G$ as functions not of $(t,\theta)$ but of $(x,v)$.
That is, we write $F$ given by \eqref{Fdef} as
$$ F(x,v) = \frac{\lambda(\lambda-1) \sigma}{2}\,  G(x,v) + \frac{(\lambda-1)(\lambda-3)}{4} \, E(x,v),$$
where $G\colon C^1(S^1)\times C^1(S^1) \to C^1(S^1)$ from equation \eqref{Gdef} and $E\colon C^1(S^1)\times C^1(S^1)\to \mathbb{R}_+$ from \eqref{Edef} are given by
$$ G(x,v)(\theta) = \int_0^{\theta} x(\phi)^{\gamma-1} v(\phi) \, d\phi - \int_0^1 x(\phi)^{\gamma} \int_0^{\phi} x(\psi)^{\gamma-1} v(\psi) \, d\psi \, d\phi$$
and
$$ E(x,v) = \gamma^2 \int_0^1 x(\phi)^{\gamma-2} v(\phi)^2 \, d\phi.$$
As long as $x$ remains strictly positive, $E$ and $G$ are smooth functions of $(x,v)$. For example, the derivative of $E$ is
$$ DE_{(x,v)}(p,q) = \gamma^2(\gamma-1) \int_0^1 x(\phi)^{\gamma-3} p(\phi) v(\phi)^2 \, d\phi + 2 \gamma^2 \int_0^1 x(\phi)^{\gamma-2} v(\phi) q(\phi) \, d\phi,$$
which depends continuously on the $C^1$ functions $(x,v,p,q)$, and further derivatives can be computed the same way. Similarly the derivative of $G$ can be computed, and for any $C^1$ functions $(x,v,p,q)$, the derivative map $DG$ will also be a $C^1$ function (actually $C^2$ since $G$ is smoothing, but we don't need that).

The only thing that remains is to check that the integral constraint
$$\int_0^1 x(\theta)^{\gamma} \,d\theta=1, \qquad \int_0^1 x(\theta)^{\gamma-1} v(\theta)\,d\theta = 0$$
is a submanifold of $C^1_+(S^1)\times C^1(S^1)$, where $C^1_+(S^1)$ denotes the $C^1$ functions on $S^1$ with strictly positive image;
this is easy by the usual implicit function theorem for Banach spaces. Then we verify that the differential equation preserves these
constraints, which is straightforward, and shows that our smooth vector field actually descends to a vector field on the submanifold.
For details about the implicit function theorem and vector fields on Banach manifolds, see for example Lang~\cite{Lang} or Abraham-Marsden-Ratiu~\cite{AMR}.
\end{proof}

The local existence proof works for any value of $\lambda$, but for global existence we only have a proof in case $\lambda=2$, because that is the case where we know conservation laws to get global bounds on solutions. Even when $\lambda=3$ we cannot prove global existence since the conservation law only applies when $\eta$ is a diffeomorphism, and by Remark \ref{2notequal3remark} we cannot expect good ODE behavior in any coordinates: even when $\sigma=0$ the equation genuinely breaks down without a unique global weak solution, since $\eta_{\theta}=x$ must go negative. But this will demonstrate that for example $x$ and $y$ cannot approach infinity.
In case $\lambda=2$ proofs were given in Deng-Chen~\cite{DC} and in T\i{\u g}lay~\cite{Tiglay2}, so we will only treat the case $\lambda=3$. The essential thing here is the formula \eqref{integrated},
which for $\lambda=3$ becomes
\begin{equation}\label{b3eulerform}
u_t + u u_{\theta} = 3\sigma Q,\qquad\text{where} \quad Q = \partial_{\theta}^{-1}(u-\sigma),
\end{equation}
with the constant of integration in $Q$ chosen so that it has mean zero, since the left side must integrate to zero.
The conservation law
\begin{equation}\label{l2conserved}
\frac{d}{dt} \int_{S^1} u(t,\theta)^2 \, d\theta = 0
\end{equation}
proved in \cite{LMT} is one of the infinite family of conservation laws for $\lambda=3$,
and although it is not very strong, it is enough to get a bound on $Q$, which allows us to control the growth of $u$ pointwise,
at least as long as $\eta$ remains a diffeomorphism and for a (possibly small) time beyond.
This strategy comes from \cite{FLQ}.

\begin{theorem}\label{globalthm}
In case $\lambda=2$, the equation \eqref{xeq} has a solution $x\colon C^{\infty}\big([0,\infty), C^1(S^1)\big)$ for any $u_0\in C^2(S^1)$. 
In case $\lambda=3$, there is an $\varepsilon>0$ such that equation \eqref{xeq} has a solution $x\colon C^{\infty}\big([0,T+\varepsilon), C^1(S^1)\big)$
for any $u_0\in C^2(S^1)$, where $T$ is the first time such that $x(T,\theta)=0$ for some $\theta$. In either case 
equation \eqref{yeq} has a solution $y$ defined on the same time interval, $[0,\infty)$ or $[0,T+\varepsilon)$. 
\end{theorem}

\begin{proof}
In the case $\lambda=3$, the transformation \eqref{xydef} simplifies to just $x(t,\theta) = \eta_{\theta}(t,\theta)$.
The easiest way to proceed is to show that $\eta$ itself satisfies a differential equation for which the right side is bounded.
Equation \eqref{xequationstep} becomes
\begin{equation}\label{etaderivativeb3}
\eta_{tt\theta}(t,\theta) = 3\sigma \big(\eta_t(t,\theta) - \sigma\big) \eta_{\theta}(t,\theta),
\end{equation}
and integrating once more in space gives
\begin{equation}\label{etattpressure}
\eta_{tt}(t,\theta) = 3\sigma P(t,\theta),
\end{equation}
where $P$ is essentially a pressure function, related to $Q$ from \eqref{b3eulerform} by $P(t,\theta) = Q(t,\eta(t,\theta))$. 
$P$ is defined uniquely by the conditions
$$ P_{\theta}(t,\theta) = (\eta_t(t,\theta) - \sigma)\eta_{\theta}(t,\theta), \qquad \int_{S^1} P(t,\theta) \eta_{\theta}(t,\theta)\,d\theta = 0.$$
Suppressing time dependence, we can write $P$ explicitly in terms of $\eta$ and $V:=\eta_t$ by
$$ P(\eta, V)(\theta) = \int_0^{\theta} \big[V(\psi)-\sigma\big] [\eta(\psi)-\eta(0)] \eta'(\psi) \, d\psi
- \int_{\theta}^1 \big[ V(\psi)-\sigma\big] \big[ \eta(1)-\eta(\psi)] \eta'(\psi) \,d\psi.$$
For periodic $\eta\in C^2(S^1)$, this defines a periodic $C^2$ function $P$ which depends smoothly on $(\eta,V)$, since it involves only products and continuous integral operators. Furthermore because there is no composition with $\eta$, this still makes sense even if $\eta$ stops being a
homeomorphism.

The $L^2$ conservation law \eqref{l2conserved}, together with the conservation of the mean from \eqref{sigmadef}, implies that
$\int_{S^1} (u-\sigma)^2 \, d\theta$ is constant in time, and in Lagrangian form this becomes
\begin{equation}\label{L2conservegeneta}
\int_{S^1} \big[V(t,\theta)-\sigma\big]^2 \eta_{\theta}(t,\theta) \, d\theta = \int_{S^1} \big[u_0(\theta)-\sigma\big]^2 \, d\theta,
\end{equation}
which again makes sense even if $\eta_{\theta}$ is not positive. As long as $\eta_{\theta}$ remains nonnegative, we obtain from the mean-zero condition the bound
\begin{align*}
\sup_{\theta\in S^1} P(\eta,V)(t,\theta) &\le \int_{S^1} \lvert P_{\theta}(t,\theta)\rvert \, d\theta 
= \int_{S^1} \lvert V(t,\theta) - \sigma\rvert \eta_{\theta}(t,\theta) \, d\theta \\
&\le \sqrt{\int_{S^1} \lvert V(t,\theta)-\sigma\rvert^2 \eta_{\theta}(t,\theta) \, d\theta }  \, \sqrt{\int_{S^1} \eta_{\theta}(t,\theta) \, d\theta}
= \sqrt{\int_{S^1} \big[u_0(\theta)-\sigma\big]^2 \, d\theta},
\end{align*}
using \eqref{L2conservegeneta} and the fact that $\eta$ is periodic.

Hence as long as $\eta_{\theta}$ remains nonnegative, we have that $P(\eta,V)$ is bounded in the $C^0$ norm uniformly in time.
Equation \eqref{etattpressure} now implies that
$\eta_{tt}$ is uniformly bounded in time, and we conclude that $V=\eta_t$ grows at most linearly in time (again as long as $\eta_{\theta}$ remains nonnegative). Equation \eqref{etaderivativeb3} now implies that
$\eta_{\theta}$ satisfies an estimate of the form
$$ \lVert \eta_{tt\theta}\rVert_{C^0} \le \big(\lVert u_0\rVert_{C^0} + Kt\big) \lVert \eta_{\theta}\rVert_{C^0}.$$
In particular the right side of the differential equation is bounded on all finite time intervals in the space of $C^1$ diffeomorphisms
$\eta$. Thus by the usual theory of ODEs in Banach spaces, e.g., Proposition 4.1.22 in \cite{AMR}, the solution can be continued for $\eta\in C^1$
as long as $\eta_{\theta}$ remains nonnegative. In particular the local existence theorem gives some small $\varepsilon>0$ such that the solution can be continued on the interval $[0,T+\varepsilon)$, beyond the time $T$ where $\eta_{\theta}$ first reaches zero.

Differentiating equation \eqref{etaderivativeb3} in $\theta$ gives, by the same reasoning, an ordinary differential equation for $\eta_{\theta\theta}$ with uniform bounds in the supremum norm; hence a $C^2$ initial condition $u_0$ leads to a $C^2$ solution $\eta$, and thus a $C^1$ solution $x$. The fact that we also have a solution $y\in C^0$ is now straightforward, since $y$ satisfies the linear ODE \eqref{yeq} with known coefficients in terms of the function $x$.
\end{proof}

This theorem establishes that the only thing that can go wrong with the global solutions of equation \eqref{densityeqn} in the cases $\lambda=2$ and $\lambda=3$ is that $\eta_{\theta}$ approaches zero. Significantly, the equation for $\lambda=3$ in the form \eqref{etaderivativeb3} depends only on $\eta$ as a function on $S^1$ of some smoothness, but \emph{not} on the fact that $\eta$ is a diffeomorphism. Hence the local existence result for the ODE holds even when $\eta_{\theta}$ reaches zero, and we get existence for some (possibly small) time beyond that. The difficulty is that without a global bound on the $L^2$ energy, we cannot extend this for all time.

Again we note that in the case $\sigma=0$ the breakdown is completely understood: when $\lambda=3$, the function $\eta$ ceases even to be a homeomorphism as $\eta_{\theta}$ becomes negative, while if $\lambda=2$ the fact that $\eta_{\theta}=x^2$ means that $\eta_{\theta}\ge 0$ always, so that typically $\eta$ will remain a homeomorphism. Since $u = \eta_t \circ\eta^{-1}$, this is the difference between the solution $u$ having shocks where it must cease being continuous, as opposed to steepening where $u$ remains continuous but its slope may approach infinity due to equation \eqref{flowderiv}. For other values of $\lambda$ things may be much worse: Sarria and Saxton~\cite{SS1} showed that for $\lambda>5$ or $\lambda<-1$, there are solutions for which $\eta_{\theta}$ approaches either zero or infinity, everywhere at the breakdown time. The reason here is that for $\lambda=2$ or $\lambda=3$, the terms in the forcing function $F$ defined by \eqref{Fdef} are well-controlled in time, while in general there are no good estimates for the growth. In the next section we will see what consequences can be found if we can obtain a global bound on the central force.

\section{Properties of central force systems with bounded forcing terms}\label{sectionboundedforce}

Bounds for the central force (not necessarily uniform, but with controlled growth in time) are crucial for what comes next. We first record the bounds we can obtain in the cases $\lambda\in\{2,3\}$, then derive some consequences that apply to any central force system (not merely those arising from Euler-Arnold equations). 

\begin{lemma}\label{forcegrowth}
For $\lambda=2$ or $\lambda=3$, the forcing function $F$ given by \eqref{Fdef} satisfies a bound
$$ \sup_{\theta\in S^1} \lvert F(t,\theta)\rvert \le \begin{cases} K^2 & \lambda=2 \\
K^2 + C t & \lambda=3
\end{cases},
$$ for all time $t\in [0,T)$ as determined by Theorem \ref{globalthm}, for some constants $K$ and $C$ depending on the initial data $u_0$.
\end{lemma}

\begin{proof}
In the case $\lambda=3$, we have already established this in the proof of Theorem \ref{globalthm}, since there
$$ F(t,\theta) = 3 \sigma G(t,\theta),$$
and $G=(\eta_t-\sigma)$ grows at most linearly in time because $\eta_{tt}$ is bounded.
In the case $\lambda=2$, the forcing function is given by
$$ F(t,\theta) = \sigma(\eta_t-\sigma) - \tfrac{1}{4} E(t),$$
and $E(t)$ is constant in time for $\lambda=2$, and given by
$$ E(t) = E(0) = \int_{S^1} u_0'(\theta)^2 \, d\theta.$$
This implies that $\int_{S^1} x_t^2 \, d\theta$ is constant in time, and we thus get a uniform bound for
$(\eta_t-\sigma)$ by the Poincar\'e inequality, since
$$ \sup_{\theta\in S^1} \lvert \eta_t-\sigma\rvert \le \int_{S^1} \lvert \eta_{t\theta}\rvert \, d\theta
= 2 \int_{S^1} \lvert x x_t\rvert \,d\theta \le \int_{S^1} x^2 \, d\theta \, \int_{S^1} x_t^2 \, d\theta,$$
and the right side is constant in time.
\end{proof}

One might hope that a polynomial-in-time bound like this is true for other values of $\lambda$; if it were,
the technique of the breakdown proof we will give later would also show the same breakdown phenomenon for
all values of $\lambda$. Ultimately the only thing we need is that the forcing function grows like a polynomial in
time, because it will be less than the exponential decay we get in general from the equation whenever $\lambda>1$.
If we could establish any kind of polynomial estimate for the energy $E(t)$ given by \eqref{Edef} for other
values of $\lambda$, we would obtain the same breakdown result here proved for $\lambda=2$ and $\lambda=3$. However the fact that
Sarria-Saxton~\cite{SS1} showed that the basic breakdown mechanism changes when $\lambda>5$ makes clear that this could
only be hoped for if $\lambda\in (1,5)$.

The main tools we use to establish breakdown are the following simple result which applies for any ODE
for fairly general forcing functions (and thus will apply here for the individual particles $x(t,\theta),y(t,\theta)$ for
each individual $\theta\in S^1$). The first lemma gives an upper bound for the solution in terms of the forcing function,
while the second establishes that solutions will eventually reach zero if their velocity is sufficiently negative.
Our philosophy is that although the forcing function depends implicitly and nonlocally on the solution for all values
of $\theta$, each individual particle feels a force $F(t)$ that is some given function of time, bounded on finite time
intervals, and thus we can treat it as essentially an external force.

\begin{lemma}\label{upperboundlem}
Suppose $\phi$ satisfies the second-order ODE
$$ \phi''(t) = F(t) \phi(t)$$
on some interval $[0,T)$, where $T$ may be infinite,
and assume $F(t)\le f(t)^2$ for some nonnegative differentiable increasing function $f$.

Then there is a $C$ such that
\begin{equation}\label{Rbound}
\frac{\phi'(t)}{\phi(t)} \le C+f(t)
\end{equation}
for all $t\in [0,T)$.
\end{lemma}

\begin{proof}
Define $R(t) = \phi'(t)/\phi(t)$. Then $R$ satisfies the Riccati inequality
\begin{equation}\label{Rinequality}
R'(t) = F(t) - R(t)^2 \le f(t)^2 - R(t)^2.
\end{equation}
If $R(t)$ is ever larger than $f(t)$, then $R(t)$ must decrease; thus if $f(0)<R(0)$, then $R(t)<R(0)$ for all time
until $R(t)$ possibly crosses $f(t)$.
If $R(t)$ is smaller than
$f(t)$, then the difference $Q(t) = f(t)-R(t)$ satisfies
$$ Q'(t) \ge f'(t) + R(t)^2 - f(t)^2 \ge f'(t) + Q(t)^2 - 2f(t) Q(t) \ge -2f(t) Q(t).$$
In particular if $Q$ is ever positive, it will always be positive. This shows
that $R(t)\le f(t)$ for all time if it is true for any time. Combining shows that
$$R(t)\le \max\{R(0), f(t)\}\le C+f(t),$$
which is equivalent to \eqref{Rbound}.
\end{proof}

\begin{lemma}\label{breakdownlemma}
Suppose
\begin{equation}\label{myODE}
\phi''(t) = F(t) \phi(t)
\end{equation}
for some continuous function $F$ on a maximal time interval $[0,T)$.
If $\phi(t_0)>0$ and $\phi'(t_0)/\phi(t_0)$ is sufficiently negative, then $\phi(t_*)=0$ for
some $t_*\in (t_0,T)$.
\end{lemma}

\begin{proof}
Let $g$ denote the solution of \eqref{myODE} satisfying
$$ g(t_0) = 1, \qquad g'(t_0)=0.$$
If $g(t)$ reaches zero in finite time, then by the Sturm comparison theorem, $\phi(t)$ must also reach zero
whenever $\phi'(t_0)/\phi(t_0)\le 0$.

Otherwise $g(t)$ is always positive, and the general solution of \eqref{myODE} is given by
$$ \phi(t) = \phi(t_0) g(t)\Big( 1 + C \int_{t_0}^t \frac{d\tau}{g(\tau)^2} \Big), \qquad C = \frac{\phi'(t_0)}{\phi(t_0)}$$
as can easily be verified by direct substitution.
(This is just reduction of order.)
The function $\phi(t)$ will turn negative for some $t$ as long as
$$ C < -1/\int_{t_0}^T \frac{d\tau}{g(\tau)^2}.$$
\end{proof}

The next result tells us about the effect of nonzero angular momentum. It is familiar from basic celestial mechanics:
even for a not-too-singular force directed toward the origin, a particle will not reach the origin if there is nonzero angular momentum,
while a particle with zero angular momentum will reach the origin in finite time. In our context this will give a lower bound on the
radial coordinate $r=\sqrt{x^2+y^2}$, which gives global existence in Theorem \ref{globalposmom} if the angular momentum is never zero.

\begin{lemma}\label{angmomboundlem}
Suppose $(x,y)$ is a planar system satisfying the ODE
\begin{equation}\label{generalsolar}
\ddot{x}(t) = F(t) x(t), \qquad \ddot{y}(t) = F(t) y(t),
\end{equation}
where $F$ is continuous and bounded on $[0,T]$. Let
\begin{equation}\label{angmom}
\omega_0 = x(0) \dot{y}(0) - y(0) \dot{x(0)} \qquad \text{and}\qquad r(t)^2 = x(t)^2 + y(t)^2.
\end{equation}
Then if $\omega_0$ is nonzero, $r(t)$ cannot reach zero on $[0,T]$.
\end{lemma}

\begin{proof}
Conservation of angular momentum ensures that
$$ x\dot{y} - y\dot{x} = \omega_0,$$
so that
$$ \dot{x}^2 + \dot{y}^2 = (x\dot{x}+y\dot{y})^2 + (x\dot{y} - y\dot{x})^2 = \dot{r}^2 + \frac{\omega_0^2}{r^2}.$$
We then obtain
$$ \frac{d}{dt} \left( \dot{r}^2 + \frac{\omega_0^2}{r^2}\right) = 2 \big(\dot{x} \ddot{x} + 2 \dot{y} \ddot{y}\big) = 2 F(t)( x\dot{x} + y\dot{y})
= 2F(t) r(t) \dot{r}(t).$$

Observe that $r(t)$ can only be made small if it is decreasing on some interval $[t_1,t_2]$, so to get an upper bound on this energy we define
$$ \overline{F} = \max\{ -\inf_{0\le t\le T} F(t), 0\}.$$
Then $-F(t) \le \overline{F}$ for all $t\in [0,T]$ and $\overline{F}\ge 0$, and
integrating over $[t_1,t_2]$ assuming that $\dot{r}(t)\le 0$ on $[t_1,t_2]$ gives
\begin{align*}
\dot{r}(t_2)^2 + \frac{\omega_0^2}{r(t_2)^2} &= \dot{r}(t_1)^2 + \frac{\omega_0^2}{r(t_1)^2} + 2 \int_{t_1}^{t_2} F(t) r(t) \dot{r}(t) \, dt \\
 &\le  \dot{r}(t_1)^2 + \frac{\omega_0^2}{r(t_1)^2} + \overline{F} \big( r(t_1)^2 - r(t_2)^2\big) \le  \dot{r}(t_1)^2 + \frac{\omega_0^2}{r(t_1)^2} +
 \overline{F}  r(t_1)^2.
 \end{align*}
In particular we obtain
$$ r(t_2) \ge \frac{\lvert \omega_0\rvert r(t_1)}{\sqrt{r(t_1)^2 \dot{r}(t_1)^2 + \omega_0^2 +
 \overline{F}  r(t_1)^4} },$$
 and in particular $r(t_2)$ is positive since $\overline{F}$ is finite by assumption.

There can only be finitely many such intervals where $r$ can decrease on $[0,T]$ since $r$ can only decrease when either $x$ or $y$
is decreasing, and a linear differential equation with bounded force coefficient can only have a discrete set of turning points in
a compact interval.
 \end{proof}

\begin{remark}\label{reallybadforce}
Of course, if we allow the forcing function to be something like $F(t) = -\frac{k^2}{(1-t)^2}$, then the particle
can reach zero in finite time. The change of time variable $s=-\ln{(1-t)}$ in this case turns each equation in the system \eqref{generalsolar}
into
$$ \frac{d^2x}{ds^2} + \frac{dx}{ds} + k^2 x = 0,$$
which will have infinitely many oscillations up to $t=1$ if and only if $k>\tfrac{1}{2}$. Thus if $k>\tfrac{1}{2}$ the system will spiral
around the origin infinitely many times until reaching the origin at $t=1$. For bounded $F(t)$, things are substantially simpler, but note
that we only have reasonable bounds on $F(t)$ in special cases (in particular $\lambda=2$ and $\lambda=3$ in the present context).
\end{remark}

One further lemma simplifies our considerations, which is the reflection symmetry of the equation \eqref{mubgeneral}--\eqref{mubICs}. Note that
since $m(t,\theta) = \sigma - u_{\theta\theta}(t,\theta)$, and $u_{\theta\theta}$ must change sign if $u$ is not constant, the condition that $m$
changes sign has somewhat different consequences for the convexity of $u$
depending on whether $\sigma$ is positive or negative. However these are illusory, and the
following proposition shows that if $\sigma\ne 0$, we can assume $\sigma>0$ without loss of generality. This proposition is well-known and appears
in many places, e.g., in \cite{FLQ}.

\begin{proposition}\label{sigmasign}
If $v(t,\theta) := -u(t,1-\theta)$, with $u$ satisfying \eqref{mubgeneral}--\eqref{mubICs}, then $v$  satisfies the equation
$$ n_t + v n_{\theta} + \lambda v_{\theta} n = 0, \qquad n = \mu(v) - v_{\theta\theta}.$$
Hence any result that applies with $\sigma=\mu(u)>0$ also applies to $v$ for $\mu(v)<0$.
\end{proposition}

\begin{proof}
Clearly if $\zeta$ denotes the reflection map $\zeta(\theta) = 1-\theta$ on the circle, then $v:=-u\circ\zeta$ satisfies $v_t = -u_t\circ\zeta$ and
$v_{\theta} = u_{\theta}\circ\zeta$. Thus we get
$$(\mu-\partial_{\theta}^2)v = -(\mu-\partial_{\theta}^2)u\circ\zeta,$$
so that if $n = \mu(v) - v_{\theta\theta}$, we have $ n = -m\circ\zeta$. This now implies $n_t = -m_t\circ\zeta$ and $n_{\theta} = m_{\theta}\circ\zeta$.
Thus composing \eqref{mubgeneral} with $\zeta$ gives
\begin{align*}
0&= m_t\circ\zeta + (u\circ\zeta) \, (m_{\theta}\circ\zeta) + \lambda (u_{\theta}\circ\zeta)\, (m\circ\zeta) \\
&= -n_t - v n_{\theta} + \lambda (v_{\theta}) (-n) = 0.
\end{align*}
This implies that $(v,n)$ satisfies the same system as $(u,m)$ in \eqref{mubgeneral}--\eqref{mubICs}. However since
$\mu(v) = -\mu(u)$, anything we may prove assuming $\mu(u)>0$ will equally apply to $v$ when $\mu(v)<0$.
\end{proof}

In light of Proposition \ref{sigmasign}, we will always assume that $\sigma>0$ without loss of generality.

\section{Proof of Theorem \ref{mainthm}}\label{sectionproofs}

First we show that if the momentum is everywhere positive or everywhere negative, then the solution of equations
\eqref{mubgeneral}--\eqref{mubICs} exists globally and gives a diffeomorphism. This result is already contained in
the original papers \cite{KLM} and \cite{LMT}, based on analytic inequalities (and generalized for any value of $\lambda$ in
\cite{tiglayvizman}), but our perspective here is different.
By Proposition \ref{sigmasign}, we may assume without loss of generality that the initial momentum is strictly positive.

\begin{theorem}\label{globalposmom}[Theorem \ref{mainthm}, ``if'' case]
If $\lambda=2$ or $\lambda=3$, and if $m_0(\theta) = \sigma - u_0''(\theta)$, with $\sigma = \mu(u_0)$, is positive for all $\theta\in S^1$,
then the solution of \eqref{mubgeneral}--\eqref{mubICs} exists for all time, and the flow $\eta$ given by \eqref{flowequation}
remains a $C^2$ diffeomorphism of the circle for all time.
\end{theorem}

\begin{proof}
By the definitions \eqref{xydef} of $x$ and $y$, the first time $x$ approaches zero, we must simultaneously have
$y$ approaching zero, since $$y = -\gamma x_{\theta} + \sigma x \int_0^t x(\tau)^{\gamma}\,d\tau.$$
Because $x$ is positive everywhere until it approaches zero, its minimum is also approaching zero, so that
$x_{\theta}$ is approaching zero at the same time; meanwhile the second term in $y$ approaches zero since
$x$ remains bounded and the integral is multiplied by $x$.
Hence the only way $\eta_{\theta} = x^{\gamma}$ can ever reach zero is if both $x$ and $y$ approach zero
simultaneously. 

Theorem \ref{globalthm} shows that for $\lambda=2$ or $\lambda=3$, the only way the solution can break down is
if $\eta_{\theta}$ reaches zero at some finite time $T$, and when this happens we still have at least local existence
in $(x,y)$ coordinates beyond this $T$. 
By Lemma \ref{angmomboundlem}, since $m_0$ is positive and $F$ is bounded by Lemma \ref{forcegrowth}, 
the quantity $x(t,\theta)^2+y(t,\theta)^2$ cannot reach zero on $[0,T]$, and we get a contradiction.
\end{proof}

Now we consider what happens when the sign of the momentum changes. By Proposition \ref{sigmasign}, we may assume
without loss of generality that $\sigma>0$. In this case, the assumption that momentum changes sign means that
$\sigma - u_0''(\theta) < 0$ for some values of $\theta\in S^1$, because it would always be true that
$\sigma - u_0''(\theta)>0$ for some values of
$\theta\in S^1$ (for example when $u_0$ has a local maximum or minimum). The important thing here becomes
$ u_0''(\theta)> \sigma,$ which in particular implies that $u_0$ is convex on some interval. This leads to a convexity
result on the function $x$, and it is on this that all our breakdown results depend.

Our strategy will be as follows: we choose points $a<b<c<d$ such that $m_0(\theta)< 0$ on $(a,d)$: then we establish that
\begin{itemize}
\item $x(t,c)$ has an upper bound independent of $t$ in Lemma \ref{step1lemma};
\item $x(t,b)/x(t,c)$ decays like $e^{-Mt}$ for some $M>0$ in Lemma \ref{step2lemma};
\item and thus $x_t(t,a)/x(t,a)$ can be made as small as we want in Lemma \ref{step3lemma},
\end{itemize}
and from this we use Lemma \ref{breakdownlemma} to show that $x$ must reach zero in finite time.
None of the choices of these points actually matter, although optimizing the choice could lead to a better estimate for the
breakdown time. All that matters is that $a$ and $d$ are chosen so that $m_0(\theta)< 0$ on $(a,d)$,
which we will assume from now on. Essentially all three lemmas rely on the same basic conservation-of-momentum equation
\begin{equation}\label{basicconsmom}
\frac{\partial}{\partial t} \left( \frac{y(t,\theta)}{x(t,\theta)}\right) = \frac{m_0(\theta)}{x(t,\theta)^2},
\end{equation}
which is a direct consequence of the equation \eqref{angularmomentum}. We apply it in three different ways: integrating in time for Lemma
\ref{step1lemma}, integrating in both time and space for Lemma \ref{step2lemma}, and integrating in space only for Lemma \ref{step3lemma}.
The first two lemmas are basically the same as arguments in the original paper of McKean~\cite{mckeanbreakdown}, while the third is a new
argument. See Figure \ref{figure2} for the heuristic in a simple case.

\begin{figure}[!ht]
\includegraphics[scale=0.25]{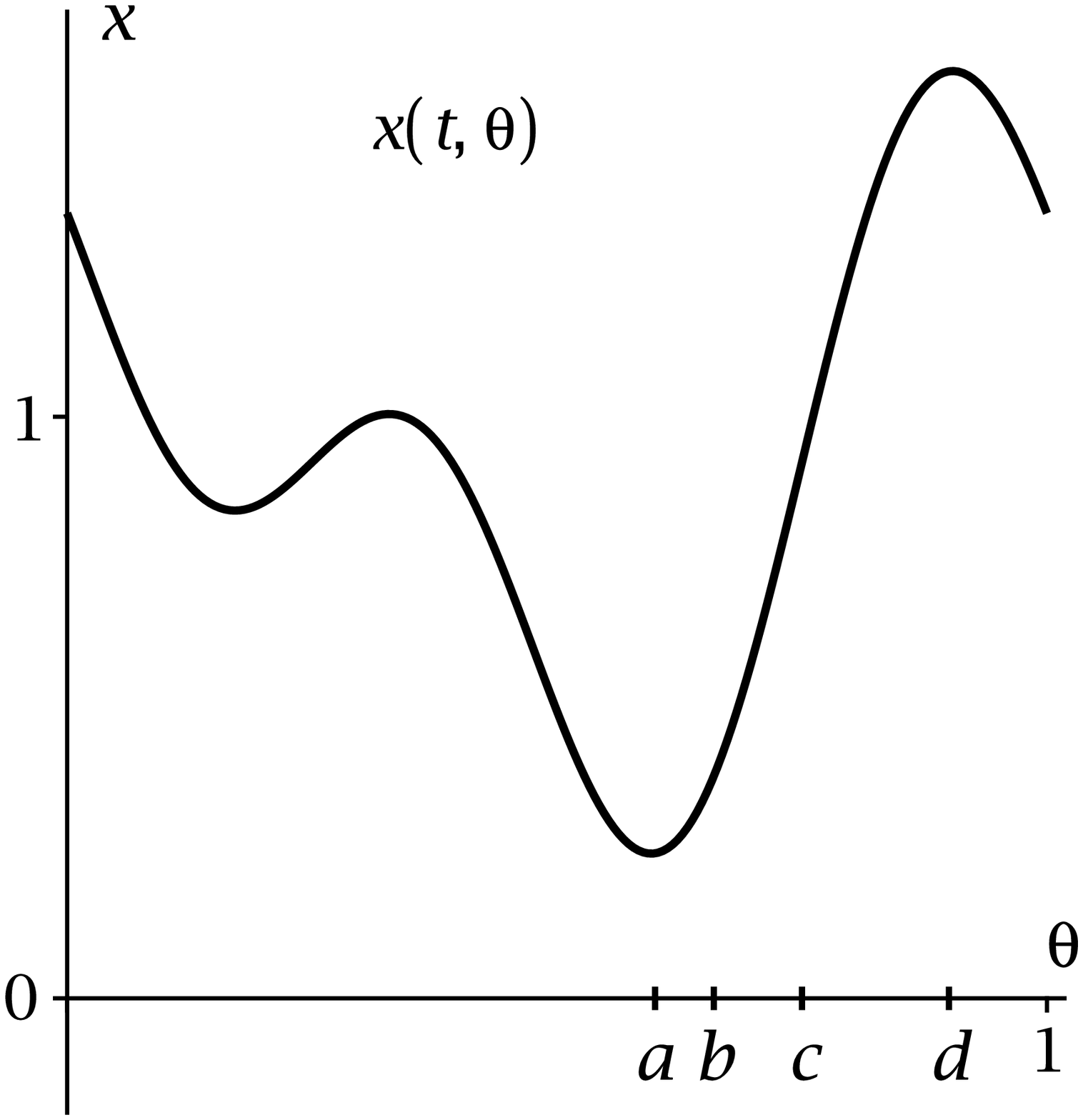} \;
\includegraphics[scale=0.25]{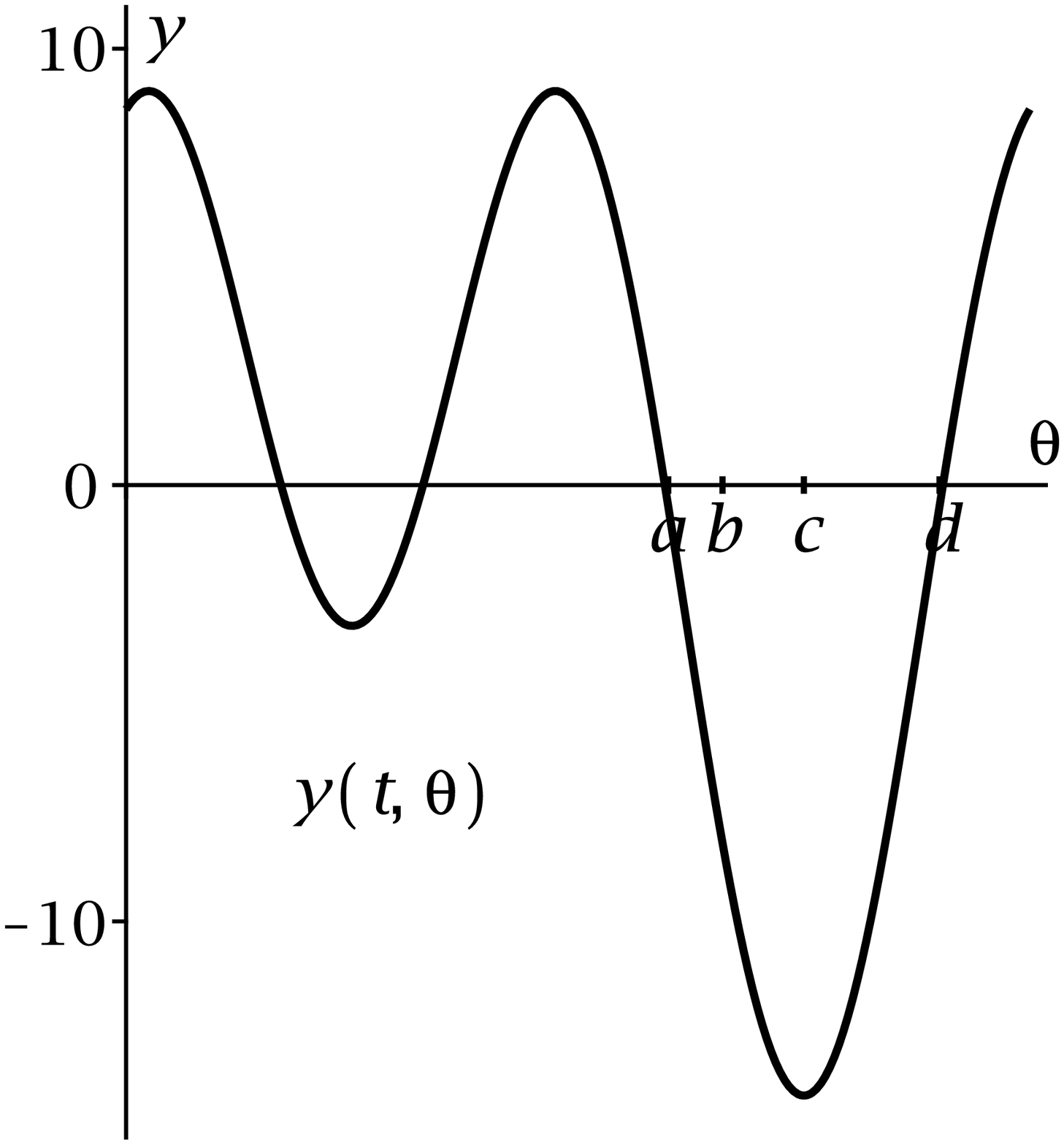} \;
\includegraphics[scale=0.25]{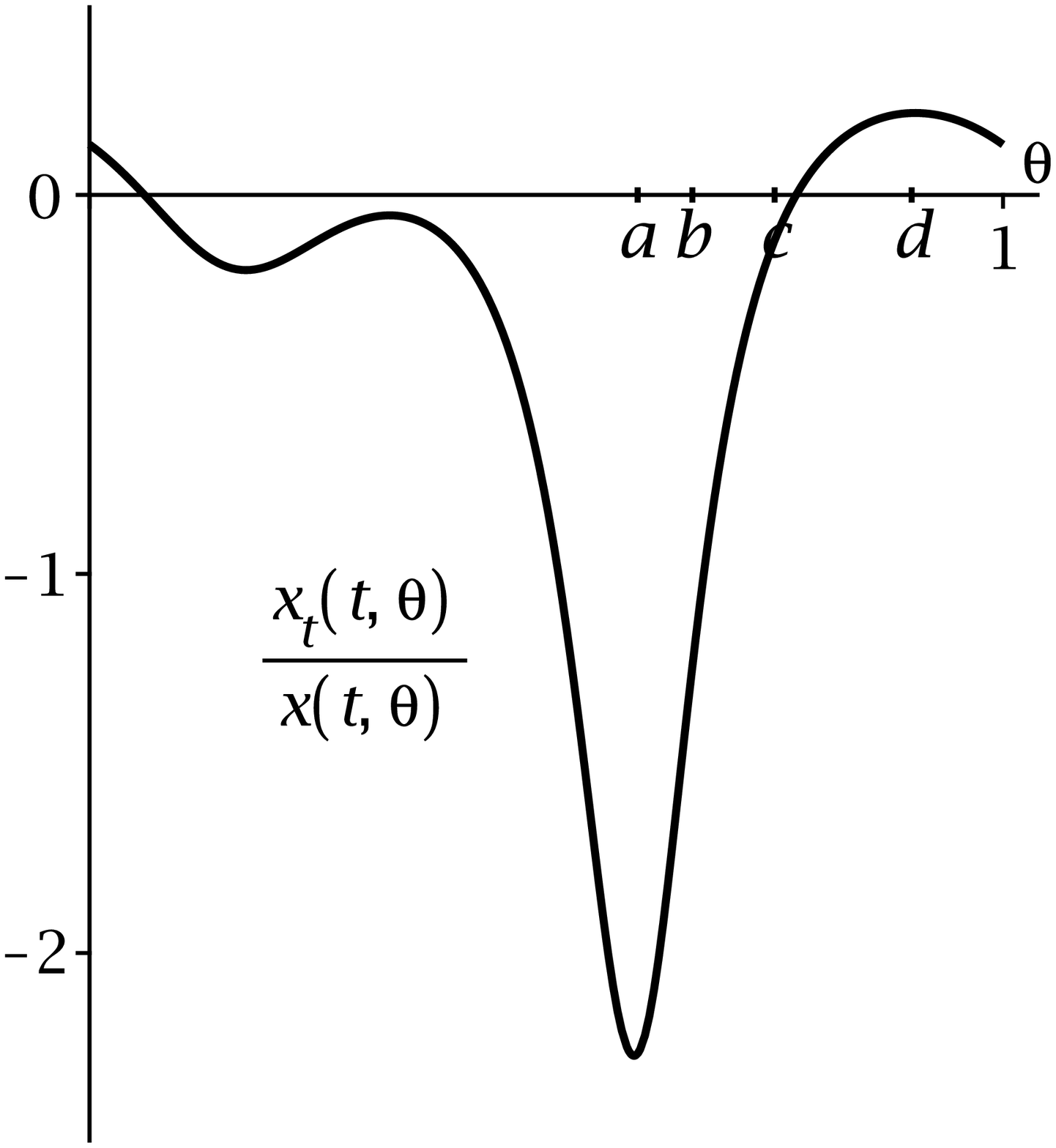}
\caption{The plots of $x$, $y$, and $x_t/x$ in the Hunter-Saxton case ($\lambda=2$ and $\sigma=0$) with $u_0(\theta) = 0.1\sin(2\pi\theta) + 0.04\cos(4\pi \theta)$ at $t=1.4$, shortly before breakdown. Note that $x$ is increasing on $(a,d)$, and $y$ is negative everywhere there, and that $x_t/x$ is most negative
at $\theta=a$. In this case $y_t/y$ is constant, so we have not plotted it.}\label{figure2}
\end{figure}

\begin{lemma}\label{step1lemma}
Suppose $\gamma>0$ and $\sigma>0$, and that $x$ and $y$ satisfy 
the equations in Theorem \ref{transformthm}, 
and thus \eqref{basicconsmom}. If $m_0(\theta)\le 0$ on the interval $[a,d]$, then for any time $t$, the function $x(t,\theta)$ is
increasing in $\theta$ for $\theta\in [a,d]$. As a consequence, we have for any $c\in[a,d]$ and any $t\ge 0$ that
\begin{equation}\label{upperboundonx}
x(t,c)\le (d-c)^{-1/\gamma}.
\end{equation}
\end{lemma}

\begin{proof}
Integrate \eqref{basicconsmom} in time to get
\begin{equation}\label{basicconsmominttime}
\frac{y(t,\theta)}{x(t,\theta)} = \frac{y(0,\theta)}{x(0,\theta)} + m_0(\theta) \int_0^t \frac{d\tau}{x(\tau,\theta)^2} = -\lvert m_0(\theta)\rvert \int_0^t \frac{d\tau}{x(\tau,\theta)^2},
\end{equation}
for all $\theta\in [a,d]$, since $y(0,\theta)=0$ everywhere and $m_0$ is nonpositive by assumption.
By the definition \eqref{xydef} of $x$ and $y$, we have
\begin{equation}\label{integratedlemma1}
-\gamma \, \frac{x_{\theta}(t,\theta)}{x(t,\theta)} + \sigma \int_0^t x(\tau,\theta)^{\gamma}\,d\tau =
-\lvert m_0(\theta)\rvert \int_0^t \frac{d\tau}{x(\tau,\theta)^2},
\end{equation}
and since $\sigma>0$ and $\gamma>0$ by assumption, we conclude that $x_{\theta}/x>0$, so that $x$ is strictly increasing as long as it remains positive.

The inequality \eqref{upperboundonx} comes from formula \eqref{xpowerintegral}. 
In particular since $x$ is increasing for $\theta\in [c,d]$, we have
$$ (d-c) x(t,c)^{\gamma} \le \int_c^d x(t,\theta)^{\gamma}\,d\theta \le \int_{S^1} x(t,\theta)^{\gamma}\,d\theta = 1,$$
which implies \eqref{upperboundonx}.
\end{proof}

The next step is to integrate equation \eqref{basicconsmominttime} over $\theta\in [b,c]$, which gives a bound on the logarithm of $x$. This implies
exponential decay in time of $x(t,b)$.

\begin{lemma}\label{step2lemma}
Consider all the same hypotheses as in Lemma \ref{step1lemma} on an interval $[a,d]$. Then for any $b,c$ with $a<b<c<d$,
the function $x$ satisfies
\begin{equation}\label{exponentialdecaybound}
x(t,b) \le x(t,c) e^{-Mt}, \qquad \text{where } M = A\sigma^{\frac{2}{\gamma+2}} \int_b^c \lvert m_0(\theta)\rvert^{\frac{\gamma}{\gamma+2}}\,d\theta,
\end{equation}
and $A$ is a constant depending only on $\gamma$.
\end{lemma}

\begin{proof}
We begin with \eqref{integratedlemma1}, in the form
\begin{equation}\label{logxderiv}
\frac{x_{\theta}(t,\theta)}{x(t,\theta)} = \int_0^t \frac{1}{\gamma}\left( \sigma x(\tau,\theta)^{\gamma} + \frac{\lvert m_0(\theta)\rvert}{x(\tau,\theta)^2}\right)\,d\tau.
\end{equation}
Elementary calculus shows that the function
$$ x\mapsto \frac{1}{\gamma}\left( \sigma x^{\gamma} + \frac{\lvert m_0\rvert}{x^2}\right)$$
is minimized among positive $x$ for $x = \left( \frac{2\lvert m_0\rvert}{\sigma \gamma}\right)^{\frac{1}{\gamma+2}}$, and the minimum value is
$$ A \lvert m_0\rvert^{\frac{\gamma}{\gamma+2}} \sigma^{\frac{2}{\gamma+2}}, \quad\text{for}\quad A = \left(\frac{2}{\gamma}\right)^{\frac{\gamma}{\gamma+2}} \left( \frac{1}{\gamma}+\frac{1}{2}\right).$$
In particular since this bound is independent of time, equation \eqref{logxderiv} implies
$$ \frac{\partial}{\partial \theta} \ln{x(t,\theta)} \ge A t  \sigma^{\frac{2}{\gamma+2}} \lvert m_0(\theta)\rvert^{\frac{\gamma}{\gamma+2}}.$$
Integrating from $\theta=b$ to $\theta=c$ gives
$$ \ln{x(t,c)} - \ln{x(t,b)} \ge Mt,$$
and exponentiation gives \eqref{exponentialdecaybound}.
\end{proof}

The last step is to use the conservation of angular momentum formula \eqref{angularmomentum}
$$ xy_t - y x_t = m_0$$
directly. Dividing through by $xy$ gives
\begin{equation}\label{xyratios}
\frac{x_t}{x} = \frac{y_t}{y} - \frac{m_0}{xy}.
\end{equation}
Now by Lemma \ref{upperboundlem}, since both $x$ and $y$ satisfy the same ODE with a bounded forcing function, the quantity
$y_t/y$ is bounded above by the square root of any increasing upper bound for
the forcing function. Meanwhile since $y$ is negative if and only if $m_0$ is, the other
term can be made as large and negative as we want when $x$ and $y$ are both small.

\begin{lemma}\label{step3lemma}
Consider the same hypotheses as in Lemma \ref{step1lemma} and \ref{step2lemma}. Then
\begin{equation}\label{intervalabbound}
\int_a^b \frac{x_t(t,\theta)}{x(t,\theta)}\,d\theta \le \int_a^b \frac{y_t(t,\theta)}{y(t,\theta)}\,d\theta - \frac{N}{x(t,b)^2},
\qquad \text{where } N = \frac{2}{\gamma} \left(\int_a^b \sqrt{\lvert m_0(\theta)\rvert}\,d\theta\right)^2.
\end{equation}
\end{lemma}

\begin{proof}
Integrating equation \eqref{xyratios} for $\theta\in [a,b]$, we obtain 
$$ \int_a^b \frac{x_t(t,\theta)}{x(t,\theta)} = \int_a^b \frac{y_t(t,\theta)}{y(t,\theta)} - J,$$
where $J$ is the positive quantity
\begin{equation}\label{Jdef}
J := \int_a^b \frac{m_0(\theta)\,d\theta}{x(t,\theta)y(t,\theta)}.
\end{equation}
We want to establish a lower bound for $J$. 

Since $m_0$ and $y$ are both negative simultaneously on $(a,b)$, the Cauchy-Schwarz inequality implies that
\begin{equation}\label{cauchyschwarz}
\left(\int_a^b \sqrt{\lvert m_0(\theta)\rvert}\,d\theta \right)^2 \le \int_a^b \frac{\lvert m_0(\theta)\rvert \,d\theta}{x(t,\theta) \lvert y(t,\theta)\rvert} \int_a^b x(t,\theta) \lvert y(t,\theta)\rvert \, d\theta.
\end{equation}
Now by formula \eqref{xydef}, and using the fact that $\lvert y\rvert = -y$ on $[a,d]$, we get
\begin{align*}
\int_a^b x(t,\theta) \lvert y(t,\theta)\rvert \, d\theta &= \gamma \int_a^b x(t,\theta) x_{\theta}(t,\theta) \,d\theta - \sigma x(t,\theta)^2 \int_0^t x(\tau,\theta)^{\gamma} \,d\tau\,d\theta \\
&\le \tfrac{\gamma}{2} \big( x(t,b)^2 - x(t,a)^2\big) \le \tfrac{\gamma}{2} x(t,b)^2.
\end{align*}
Now plug this inequality into \eqref{cauchyschwarz} to get that $J$ given by \eqref{Jdef} satisfies
$$ J \ge \frac{2}{\gamma x(t,b)^2} \left( \int_a^b \sqrt{\lvert m_0(\theta)\rvert} \,d\theta\right)^2.$$
This then yields \eqref{intervalabbound}.
\end{proof}

Combining Lemmas \ref{step1lemma}--\ref{step3lemma}, we can now prove the second half of Theorem \ref{mainthm}.
Everything here would in fact work for any value of $\lambda>1$, not just $\lambda=2$ or $\lambda=3$, except for the fact that we need
a subexponential upper bound for the forcing function in order to use Lemma \ref{upperboundlem}.

\begin{theorem}\label{breakdownthm}[Theorem \ref{mainthm}, ``only if'' case]
Suppose $\sigma>0$ and that $\lambda=2$ or $\lambda=3$.
If the sign of $m_0 = \sigma - u_0''$ changes on the circle, then $C^2$ solutions of
\eqref{mubgeneral}--\eqref{mubICs} must break down in finite time, as the Lagrangian flow
given by \eqref{flowequation} ceases to be a diffeomorphism.
\end{theorem}

\begin{proof}

Choose any subdivision $a<b<c<d$ such that $m_0$ is negative on $(a,d)$, and such that $m_0(a)=0$. Lemma \ref{step1lemma} implies that
$$ x(t,c) \le (d-c)^{-1/\gamma}.$$
Lemma \ref{step2lemma} then implies that
$$ x(t,b) \le x(t,c) e^{-Mt} \le (d-c)^{-1/\gamma} e^{-Mt},$$
where $M>0$ is given by equation \eqref{exponentialdecaybound}. Applying Lemma \ref{step3lemma} then gives
$$ \int_a^b \frac{x_t(t,\theta)}{x(t,\theta)}\,d\theta \le \int_a^b \frac{y_t(t,\theta)}{y(t,\theta)}\,d\theta -
N (d-c)^{2/\gamma} e^{2Mt},$$
where $N>0$ is given by \eqref{intervalabbound}.

Since $y$ satisfies the equation $y_{tt}(t,\theta)= F(t,\theta) y(t,\theta)$ by Theorem \ref{transformthm},
the quantity $y_t/y$ is bounded above by an estimate of the form
\begin{equation}\label{ytratio}
\frac{y_t(t,\theta)}{y(t,\theta)} \le C(\theta) + f(t,\theta),
\end{equation}
where $f(t,\theta)$ is any positive increasing function satisfying $F(t,\theta)\le f(t,\theta)^2$ for all $t$ and $\theta$, as in
Lemma \ref{upperboundlem}. If $\lambda=2$ or $\lambda=3$, we can use Proposition \ref{forcegrowth} to see that
$f(t,\theta)$ grows at most polynomially in time, for each value of $\theta$, and this implies by
Lemma \ref{upperboundlem} that $y_t(t,\theta)/y(t,\theta)$ grows at most polynomially in time.
Integrating over the interval $\theta\in [a,b]$ still gives polynomial growth in time, and this implies that our estimate takes the form
$$ \int_a^b \frac{x_t(t,\theta)}{x(t,\theta)} \,d\theta \le P(t) - N (d-c)^{2/\gamma} e^{2Mt},
$$
where $P(t)$ is a function growing at most like a power of $t$. Since the exponential term eventually dominates, we see that
we can make the integral
$$ \int_a^b \frac{x_t(t,\theta)}{x(t,\theta)} \,d\theta$$
as small as we want, which also implies that for some $\theta\in [a,b]$, the quantity $x_t(t,\theta)/x(t,\theta)$ can be made
as small as desired. For such $\theta$, Lemma \ref{breakdownlemma} implies that $x(t,\theta)$ must reach zero in finite time.
Of course, since $x(t,\cdot)$ is increasing on $[a,d]$, the smallest value must occur at $\theta=a$, when the sign of $m_0$
changes from positive to negative.
\end{proof}

\section{Outlook}\label{sectionoutlook}

The general principle that $m_0>0$ or $m_0<0$ everywhere implies global existence of classical solutions for solutions
of \eqref{mubgeneral} is established in T\i{\u g}lay-Vizman~\cite{tiglayvizman} as long as the definition of $m$ in terms of $u$
that replaces \eqref{momentumdef} involves at least two derivatives of $u$. In many situations of interest, the operator
$m$ has mean zero for all $u$, and so it is impossible for $m_0$ to have a constant sign; thus we would expect all classical
solutions to break down in finite time. As an example we return to the Okamoto-Sakajo-Wunsch equation~\cite{OSW}, given by \eqref{mubgeneral} where $m = Hu_{\theta}$, for which $m$ integrates to zero, and it is impossible to have $m_0$ positive or negative everywhere. (On the real
line the situation is different, but our periodic context forecloses such possibilities.)

The following construction was presented in \cite{BKP} in the case $\lambda=2$, but most things work the same way for any value of $\lambda$. Breakdown
for all solutions in the case $\lambda=2$ was given in \cite{washabaughpreston}, while breakdown for all positive $\lambda$ with $u_0$ odd was given
by Castro-Cord\'oba~\cite{CC}.  For $\lambda>0$, all solutions break down in finite time, while for $\lambda<0$ the solution is much more complicated and unknown in general (particularly in the most important case $\lambda=-1$, the De Gregorio equation). For the state of the art on global existence and breakdown for such equations, see Chen~\cite{chen} for the periodic case, Elgindi-Jeong~\cite{elgindijeong} for the nonperiodic case, and references in both.

\begin{proposition}\label{outlooktheorem}
Suppose $u$ and $m$ satisfy \eqref{mubgeneral} with momentum defined by $m = Hu_{\theta}$, i.e., the modified Constantin-Lax-Majda equation. 
Define the transformation
\begin{equation}\label{mclmtransformation}
x = \eta_{\theta}^{\lambda/2} \cos{\psi}, \qquad y = \eta_{\theta}^{\lambda/2} \sin{\psi},
\end{equation}
where $\psi$ is defined by 
\begin{equation}\label{psidef}
\psi(t,\theta) = \frac{\lambda m_0(\theta)}{2} \int_0^t \frac{d\tau}{\eta_{\theta}(\tau,\theta)^{\lambda}}.
\end{equation}
Then $(x,y)$ satisfy a solar model of the form 
$$ x_{tt}(t,\theta) = - \frac{\lambda}{2} F\big(t,\eta(t,\theta)\big) x(t,\theta), \qquad  y_{tt}(t,\theta) = - \frac{\lambda}{2} F\big(t,\eta(t,\theta)\big) y(t,\theta),$$
where $F(t,\theta)$ is always positive. 
\end{proposition}

\begin{proof}
As in \cite{BKP}, we start with
\begin{equation}\label{mCLMeq}
m_t + u m_{\theta} + \lambda u_{\theta} m = 0, \qquad m = Hu_{\theta},
\end{equation}
and applying the Hilbert transform gives
$$ u_{t\theta} + u u_{\theta\theta} - \frac{\lambda}{2} (m^2 -u_{\theta}^2) = -F, \qquad F = -uu_{\theta\theta} - H(uHu_{\theta\theta}),$$
using the product identity. For any $u$, the function $F$ is positive at every point, as shown in \cite{BKP}.
In Lagrangian form using \eqref{flowequation}, \eqref{vorticitytransport}, and \eqref{flowderiv}, this becomes
$$ \frac{\partial}{\partial t}\left( \frac{\eta_{t\theta}}{\eta_{\theta}}\right) + \frac{\lambda}{2} \left( \frac{\eta_{t\theta}}{\eta_{\theta}}\right)^2  =
\frac{\lambda}{2} \frac{m_0^2}{\eta_{\theta}^{2\lambda}} - F(t,\eta).$$
The transformation $\rho = \eta_{\theta}^{\lambda/2}$ turns this into the Ermakov-Pinney-type equation
\begin{equation}\label{OSWradial}
\rho_{tt} = \frac{\lambda^2}{4} \, \frac{m_0^2}{\rho^3} - \frac{\lambda}{2} \, F \rho.
\end{equation}

The usual theory of the Ermakov-Pinney equation shows how to linearize \eqref{OSWradial}: we define functions 
$x = \rho \cos{\psi}$ and $y=\rho \sin{\psi}$ for some function $\psi$, and we easily compute that 
$$ x_{tt} = - \frac{\lambda}{2} \, F x \qquad \text{and}\qquad y_{tt} = - \frac{\lambda}{2} \, F y$$ 
is satisfied if and only if $\psi$ satisfies 
$$ \rho \psi_{tt} + 2 \rho_t \psi_t = 0.$$
Integrating this in time gives equation \eqref{psidef}. 
\end{proof}


This formulation makes it obvious that if $\lambda>0$, the force is attracting, and zero angular momentum with
$y(0,\theta)=0$ and $x_t(0,\theta)<0$ implies $\rho(t,\theta)$ reaches zero in finite time. Hence $\eta_{\theta}$ does as well.
(There is always such a $\theta\in S^1$ by the Hopf Lemma; see \cite{washabaughpreston}.)

If $\lambda<0$, the effective force in the solar model becomes repulsive. The singular condition for $\lambda<0$ is no longer that 
$\eta_{\theta}\to 0$, but rather that $\eta_{\theta}\to \infty$. This again translates into $\rho\to 0$. (This corresponds to
$u_{\theta}$ approaching positive infinity rather than negative infinity.) It is still possible that the
particle can approach the origin, but it needs to have both zero angular momentum and a sufficiently negative
velocity pointing toward the origin to counteract the repulsive force.

We give a simple example of a bound that is straightforward in the solar model.

\begin{corollary}\label{degregoriocorollary}
Suppose $\lambda=-1$ and $u$ and $m$ satisfy \eqref{mCLMeq}. If $\theta\in S^1$ is such that $m_0(\theta)\ne 0$, then 
\begin{equation}\label{etathetamax}
\eta_{\theta}(t,\theta) \le 1  + \frac{u_0'(\theta)^2}{m_0(\theta)^2}
\end{equation}
for every $t\ge 0$ as long as the solution exists. 
\end{corollary}

\begin{proof}
In case $\lambda=-1$, equation \eqref{OSWradial} takes the form
$$ \rho_{tt} = \frac{m_0^2}{4\rho^3} + \frac{1}{2} \, F\rho.$$
Positivity of $F$ means that $\rho_{tt}$ is strictly positive, and this implies that while $\rho$ may possibly decrease on some interval $[0,t_0]$, it
must eventually increase, and once it begins to increase it must continue. 

If for some $\theta$ we know that $\rho(t,\theta)$ is decreasing on $[0,t_0]$ and increasing for $t>t_0$, then we compute (at fixed $\theta$) that
$$ \frac{d}{dt} \left( \rho_t^2 + \frac{m_0^2}{4\rho^2}\right) = 2\rho_t \rho_{tt} - \frac{m_0^2\rho_t}{2\rho^3} = F \rho \rho_t.$$
On $[0,t_0]$ the right side is nonpositive, and we obtain 
$$ \rho_t(t_0,\theta)^2 + \frac{m_0(\theta)^2}{4\rho(t_0,\theta)^2} \le \rho_t(0,\theta)^2 + \frac{m_0(\theta)^2}{4\rho(0,\theta)^2} = \frac{u_0'(\theta)^2 + m_0(\theta)^2}{4}.$$
In particular we have 
$$ \rho(t_0,\theta)^2 \ge \frac{m_0(\theta)^2}{u_0'(\theta)^2 + m_0(\theta)^2}.$$
Since $\rho$ must continue to increase for $t\ge t_0$, this is indeed the minimum possible value of $\rho(t,\theta)$ on the maximum time interval of existence.

Since $\eta_{\theta} = \frac{1}{\rho^2}$, we conclude that $\eta_{\theta}$ is bounded above by 
$$ \eta_{\theta}(t,\theta) \le \eta_{\theta}(t_0,\theta) = \frac{1}{\rho(t_0,\theta)^2} \le 1 + \frac{u_0'(\theta)^2}{m_0(\theta)^2},$$
on the maximum time interval of existence. 
\end{proof}
 
Obviously Corollary \ref{degregoriocorollary} is only useful when $m_0(\theta)\ne 0$, and by definition of our momentum operator $m = Hu_{\theta}$,
there will certainly be points where $m_0=0$. However such estimates could be useful for estimating the forcing function
$F$, which depends nonlocally on our variables. (Note that bounds on $F$ were derived in \cite{washabaughpreston}.) We leave further analysis for future research, but the point is that the
general framework here relates a family of Euler-Arnold-type PDEs to a well-understood central force system, which makes some
phenomena regarding breakdown or global existence easier to intuitively understand.

The reason this approach works is because the equations are ``nearly'' linear in terms of the variable $\eta_{\theta}$. Of
course the coefficients of this equation depend on $\eta_{\theta}$, and a transformation may eliminate some of this dependence
(e.g., quadratic terms like $\eta_{t\theta}^2/\eta_{\theta}^2$ can be eliminated by a power transformation). This is due to the fact 
that $\eta$ satisfies some kind of geodesic equation of the form $\eta_{tt} + \Gamma(\eta; \eta_t,\eta_t) = 0$
for some Christoffel map $\Gamma$, which is bilinear and symmetric in the last two variables but typically depends in a complicated
way on the first. Differentiating this with respect to any parameter leads to the Jacobi equation for  the variation.
In infinite dimensions the spatial variable $\theta$ itself can always be treated as this variational parameter, so that $\eta_{\theta}$
always satisfies the Jacobi equation. The coefficients and covariant derivative here depend on $\eta$ (and thus indirectly on $\eta_{\theta}$),
so we cannot view this as a true linear equation, but if the curvature is bounded or well-understood, this equation may be easy to analyze.
These are the situations we have studied here. The fact that equation \eqref{mubgeneral} applies to many situations of
continuum mechanics suggests that this technique may produce new insights that are not obvious from direct PDE techniques.

The author states that there is no conflict of interest. No data was produced for this paper.

\end{document}